\documentclass[twoside,11pt]{article}

\usepackage{blindtext}
\usepackage{notation}
\RequirePackage{amsthm,amsmath}

\RequirePackage{graphicx}
\RequirePackage{epsfig}
\RequirePackage{amssymb}
\RequirePackage{graphicx}
\RequirePackage{caption}
\RequirePackage{subcaption}
\RequirePackage{tikz}
\usepackage{times}
\usepackage{amsfonts}
\usepackage{amsmath, nccmath}
\usepackage{geometry}
\usepackage{comment}
\usepackage{xcolor}
\usepackage{enumitem}
\usepackage{color}
\usepackage{microtype}
\usepackage{graphicx}
\usepackage{bbold}
\usepackage{xfrac}
\usepackage{thmtools}
\usepackage[ruled]{algorithm2e}
\usepackage{booktabs}
\usepackage{mathtools}
\usepackage{thm-restate}
\usepackage{algpseudocode}

\usepackage{jmlr2e}

\usepackage{lastpage}

\ShortHeadings{Estimating the Minimizer and the Minimum Value of a Regression Function}{Akhavan, Gogolashvili and Tsybakov}
\firstpageno{1}

\begin{document}

\title{Estimating the Minimizer and the Minimum Value of a Regression Function under Passive Design}

\author{\name Arya Akhavan \email aria.akhavanfoomani@iit.it \\
      \addr Istituto Italiano di Tecnologia\\
      CREST, ENSAE, IP Paris
      \AND
      \name Davit Gogolashvili \email davit.gogolashvili@eurecom.fr \\
      \addr
      Data Science Department, EURECOM, France
      \AND
      \name Alexandre B. Tsybakov \email alexandre.tsybakov@ensae.fr \\
      \addr CREST, ENSAE, IP Paris
      }

\editor{ }

\maketitle

\begin{abstract}
We propose a new method for estimating the minimizer $\bx^*$ and the minimum value $f^*$ of a smooth and strongly convex regression function $f$ from the observations contaminated by random noise. Our estimator $\bz_n$ of the minimizer $\bx^*$ is based on a version of the projected gradient descent with the gradient estimated by a regularized local polynomial algorithm. Next, we propose a two-stage procedure for estimation of the minimum value $f^*$ of regression function $f$. At the first stage, we construct an accurate enough estimator of $\bx^*$, which can be, for example, $\bz_n$. At the second stage, we estimate the function value at the point obtained in the first stage using a rate optimal nonparametric procedure. We derive non-asymptotic upper bounds for the quadratic risk and optimization risk of $\bz_n$, and for the risk of estimating $f^*$. We establish minimax lower bounds showing that, under certain choice of parameters, the proposed algorithms achieve the minimax optimal rates of convergence  on the class of smooth and strongly convex functions.
\end{abstract}


\section{Introduction}
Estimating the minimum value and the minimizer of an unknown function from observation of its noisy values on a finite set of points is a key problem in many applications. 
Let $D = \{\bx_1,\dots,\bx_n\}\subset \mathbb{R}^d$ be a  design set and let $\com$ be a compact and convex subset of $ \mathbb{R}^d$. Assume that we observe noisy values of an unknown regression function $f:\mathbb{R}^d\to \mathbb{R}$ at points of the design set:
\begin{equation}\label{mainmodel}
    y_i = f(\bx_i)+\xi_i,  \quad i=1,\dots,n,
\end{equation}
where $\xi_i$'s are independent zero mean errors. Our goal is to estimate the minimum value of the regression function $f^* = \min_{\bx\in \com}f(\bx)$ and its location $\bx^* = \argmin_{\bx\in \com}f(\bx)$ when $\bx^*$ is unique. As accuracy measures of an estimator $\hat \bx_n$  of $\bx^*$ we consider the optimization risk $\Exp[f(\hat \bx_n) - f^*]$ and the quadratic risk 
$\Exp[\norm{\hat \bx_n - \bx^*}^2]$, where $\Vert\cdot\Vert$ denotes the Euclidean norm. The accuracy of an estimator $T_n$ of $f^*$ will be measured by the risk $\Exp\vert T_n- f^*\vert$. We will assume that $f$ belongs to the class of $\beta$-H{\"o}lder smooth and strongly convex functions with $\beta\ge 2$ (see Section \ref{sec:definitions} for the definitions). 

The existing literature considers two different assumptions on the choice of the design. Under the {\em passive} design setting, the points $\bx_i$ are sampled independently from some probability distribution. Under the {\em active} (or sequential) design setting, for each $i$ the statistician can plan the experiment by selecting the point $\bx_i$ depending on the previous queries and the corresponding responses $\bx_1,y_1,\dots, \bx_{i-1},y_{i-1}.$ The accuracy of estimation under the active design is at least as good as under the passive design but it can be strictly better, which is the case for the problems considered here. 

{\bf Active design, estimation of $\bx^*$.}
Active (or sequential) scheme has a long history  starting at least from the seminal work of \cite{kiefer1952stochastic} where an analog of the Robbins-Monro algorithm was introduced to estimate the minimizer $\bx^*$ of a univariate function $f.$ The idea of the Kiefer-Wolfowitz (KW) method is to approximate the derivative of $f$ using first order differences of $y_i$'s and plug this estimator in the gradient algorithm. \cite{kiefer1952stochastic} proved convergence in probability of the KW algorithm under some regularity conditions on the regression function. A multivariate extension of the KW algorithm was proposed by \cite{blum1954multidimensional}. Convergence rates of the KW algorithm for $d=1$ were investigated in \cite{dupavc1957kiefer} proving an upper bound on the quadratic risk of the order $n^{-2/3}$ for $\beta=3$. 
By using suitably chosen linear combinations of first order differences to approximate the gradient, \cite{fabian1967stochastic}  proved the existence of a method that attains, for odd integers $\beta\ge 3$, the quadratic risk of the order $n^{-(\beta-1) /\beta}$ for functions $f$ with bounded $\beta$th partial derivatives. The method of \cite{fabian1967stochastic} uses $(\beta-1)/2$ evaluations $y_i$ at every step  of the algorithm in order to approximate the gradient.  \cite{chen1988lower} and \cite{PT90} have established minimax lower bounds for the estimation risk on the class of $\beta$-H{\"o}lder smooth and strongly convex functions $f$,  for all $\beta \geq 2.$ For the quadratic risk, these bounds are of the order  $n^{-(\beta-1)/\beta}$. \cite{PT90} proposed a new class of  methods using smoothing kernels and randomization to approximate the gradient. This constitutes an alternative to the earlier used deterministic schemes derived from finite differences. \cite{PT90} proved that such randomized methods attain the minimax optimal rate  $n^{-(\beta-1)/\beta}$ on the above classes for all $\beta \geq 2$ and not only for odd integers $\beta\ge3$.  An  additional advantage over Fabian's algorithm is the computational simplicity of these methods. In particular, they require at each step only one or two  evaluations of the function. For subsequent developments on similar methods, we refer to  \cite{dippon2003accelerated,BP2016,akhavan2020, akhavan2021distributed,akhavan2023}, where one can find further references.  

{\bf Active design, estimation of $f^*$.} The problem of estimating $f^*$ under the active scheme was first considered by
\cite{mokkadem2007companion} who suggested a recursive estimator and proved its asymptotic normality with $\sqrt{n}$ scaling. \cite{belitser} defined an estimator of $f^*$ via a multi-stage procedure whose complexity increases exponentially with the dimension $d$, and showed that this estimator achieves  (asymptotically, for $n$ greater than an exponent of $d$) the $O_p(1/\sqrt{n})$ rate when $f$ is $\beta$-H{\"o}lder and strongly convex with $\beta> 2$. \cite{akhavan2020} improved upon this result by constructing a simple computationally feasible estimator $\hat f_n$ such that $\Exp\vert\hat f_n- f^*\vert=O(1/\sqrt{n})$ for $\beta\geq 2$. It can be easily shown that the rate $1/\sqrt{n}$ cannot be further improved when estimating $f^*$. Indeed, using the oracle that puts all the queries at the unknown true minimizer $\bx^*$ one cannot achieve better rate under the Gaussian noise. 

{\bf Passive design, estimation of $\bx^*$.}
The problem of estimating the minimizer  $\bx^*$  under the i.i.d. passive design was probably first studied in \cite{hardle1987nonparametric}, where some consistency and asymptotic normality results were discussed. \cite{tsybakov1990passive} proposed to estimate $\bx^*$ by a recursive procedure using local polynomial approximations of the gradient.  Considering the class of  strongly convex and  $\beta$-H{\"o}lder ($\beta \geq 2$) regression functions $f$, \cite{tsybakov1990passive} proves that the minimax optimal rate of estimating $\bx^*$ on the above class of functions is $n^{-(\beta-1)/(2\beta+d)}$, and shows that the proposed estimator attains this optimal rate. However, in order to define this estimator, one needs  to know of the marginal density of the design points that may be inaccessible in practice. 

There was also some work on estimating $\bx^*$ in different passive design settings.  Several papers  are analyzing estimation of $\bx^*$ in a passive scheme, where $\bx_i$'s are given non-random points in $[0,1]$ (\cite{muller1985kernel,muller1989adaptive}) or in $[0,1]^d$  (\cite{facer2003nonparametric}). Another line of work (\cite{hardle1987nonparametric, nazin-polyak-tsybakov-1989,tsybakov1990passive, nazin-polyak-tsybakov}) is to consider the problem of estimating the zero of a nonparametric regression function under i.i.d. design, also called passive stochastic approximation when recursive algorithms are used. \cite{nazin-polyak-tsybakov-1989,tsybakov1990passive, nazin-polyak-tsybakov} establish minimax optimal rates for this problem and propose passive stochastic approximation algorithms attaining these rates.  Application to transfer learning is recently developed in \cite{passive2022}, where one can find further references on passive stochastic approximation.

{\bf Passive design, estimation of $f^*$.}
To the best of our knowledge, the problem of estimating $f^*$ under i.i.d. passive design was not studied. However, there was some work on a related and technically slightly easier problem of estimating the maximum of a function observed under the Gaussian white noise model in dimension $d=1$ (\cite{Ibra,Lepski1993}). Extrapolating these results to the regression model and general $d$ suggests that the optimal rate of convergence for estimating $f^*$ on the class of $\beta$-H{\"o}lder regression functions $f$ is of the order $(n/\log n)^{-\beta/(2\beta+d)}$. It is stated as a conjecture in \cite{belitser-correction} for the passive model with equidistant deterministic design. It remains unclear whether this conjecture is true since, for higher dimensions, the effect of the equidistant grid induces an additional bias. However, we prove below that, under the i.i.d. random design, the minimax optimal rate on the class of  $\beta$-H{\"o}lder functions (without strong convexity) is indeed $(n/\log n)^{-\beta/(2\beta+d)}$. 
We are not aware of any results on estimation of $f^*$  
on the class of  $\beta$-H{\"o}lder and strongly convex   regression functions $f$, which is the main object of study in the current work.

Finally, we review some results on a related problem  of estimating the mode of a probability density function. There exists an extensive literature on this problem.  In the univariate case, \cite{parzen1962estimation} proposed the maximizer of kernel density estimator (KDE) as an estimator for the mode. Direct estimate of the mode based on order statistics was proposed by \cite{grenander1965some}, where the consistency of the proposed method was shown. Other estimators of the mode in the univariate case were considered by \citep{chernoff1964estimation,dalenius1965mode,venter1967estimation}.  The minimax rate of mode estimation  on the class of $\beta$-H{\"o}lder densities that are strongly concave near the maximum was shown to be $n^{-(\beta-1)/(2\beta+d)}$ in  \cite{tsybakov1990mode}, where the optimal recursive algorithm was introduced. It generalizes an earlier result of \cite{hasminskii1979lower} who considered the special case $d=1, \beta =2$ and derived the minimax lower bound of the order $n^{-1/5}$ matching the upper rate provided by \cite{parzen1962estimation}.  \cite{klemela2005adaptive} proposed to use the maximizer of KDE with the smoothing parameter chosen by the Lepski method \citep{lepskii1991problem}, and showed that this estimator  achieves optimal adaptive rate of convergence.
\cite{dasgupta2014optimal} proposed minimax optimal estimators of the mode based on $k$-nearest neighbor density estimators, emphasizing the implementation ease of the method. Computational complexity of mode estimation was investigated by  \cite{arias2022estimation} showing the impossibility of a minimax optimal algorithm with sublinear computational complexity. It was shown that the maximum of a histogram, with a proper choice of bandwidth, achieves the minimax rate while running in linear time.
Bayesian approach to the mode estimation was developed by \cite{yoo2019bayesian}.

\paragraph{Contributions.} In this paper, we consider the model described at the beginning of this section under the i.i.d. passive observation scheme.  The contributions of the present work can be summarized as follows.   
\begin{itemize}
\item Assuming that $f$ belongs to the class of $\beta$-H{\"o}lder and strongly convex regression functions we construct a recursive estimator of the minimizer $\bx^*$ adaptive to the unknown marginal density of $\bx_i$'s and achieving the minimax optimal rate $n^{-(\beta-1)/(2\beta+d)}$, for $\beta\ge 2$, up to a logarithmic factor. 
\item We show that the minimax optimal rate for the problem of estimating the minimum value $f^*$ of function $f$ on the above class of functions scales as $n^{-\beta/(2\beta+d)}$, and we  propose an algorithm achieving this optimal rate for $\beta> 2$. 
\item 
We prove that the minimax optimal rate of estimating $f^*$ on the class of  $\beta$-H{\"o}lder functions (without strong convexity) is of the order $(n/\log n)^{-\beta/(2\beta+d)}$. 
Thus, dropping the assumption of strong convexity causes a deterioration of the minimax rate only by a logarithmic factor. It suggests that strong convexity is not a crucial advantage in estimation of the minimum value of a function under the passive design. 
\end{itemize}
Given our results, we have the following table summarizing the minimax optimal rates for estimation under the active and passive design.   
\begin{table}[h]
\centering
\begin{tabular}{@{}lllll@{}}
\toprule
                       & $\text{rate of quadratic risk, estimation of} \,\, \bx^*$                                                   & ${}$  & ${\text{rate of estimating }f^*}\phantom{1000000000}$ \\ \midrule
\texttt{$\phantom{10}$ passive scheme} & $\phantom{10000000}{n^{-\frac{2(\beta-1)}{2\beta+d}}}$   & ${}$   & $\phantom{1000}n^{-\frac{\beta}{2\beta+d}}$ \\[.3cm]
\texttt{$\phantom{10}$ active scheme} & $\phantom{10000000}n^{-\frac{\beta-1}{\beta}}$       & ${}$ & $\phantom{100000}n^{-\frac{1}{2}}$ 
\end{tabular}
\caption{Comparisons between the rates of convergence for passive and active schemes {in the class of $\beta$-Hölder and strongly convex functions}}\label{table:comparisons}
\end{table}
We note that the convergence rates for the passive scheme suffer from the curse of dimensionality, while the rates for the active scheme are independent of the dimension. 

\paragraph{Notation.} In all the theorems, where the rates contain $\log(n)$, we assume that $n\ge 2$.
We denote by ${\mathbf E}_f$ the expectation with respect to the distribution of $(\bx_i,y_i)_{i=1}^n$ satisfying the model \eqref{mainmodel}; we also abbreviate this notation to ${\mathbf E}$ when there is no ambiguity.
Vectors are represented by bold symbols while uppercase English letters are used to denote matrices. We  denote by $\|\cdot\|$ the Euclidean norm, and by $\left\|\cdot\right\|_{\op}$ the operator norm, i.e., for a matrix $A$ we have $\left\|A\right\|_{\op} = \sup_{\|\bu\| \leq 1} \left\|A \bu \right\|.$ We denote the smallest eigenvalue of a square matrix $U$ by $\lambda_{\min}\left(U\right)$. For any $m \in \mathbb{N}$, we denote by $[m]$ the set that contains all positive integers $k$, such that $1\leq k\leq m$. For $\beta \in \mathbb{R}_{+},$  let $\lfloor\beta\rfloor$ be the biggest integer smaller than 
$\beta.$ Let $S$ denote the number elements in the set $\{\bm:|\bm| \leq \ell\}$, where $\bm$ is a $d$-dimensional multi-index. For $\bu \in \mathbb{R}^d$, let $U(\bu) = \left(\frac{\bu^{\bm^{(1)}}}{\bm^{(1)}},\dots, \frac{\bu^{\bm^{(S)}}}{\bm^{(S)}}\right)^{\top},$ where the numbering is such that $\bm^{(1)}=(0, \ldots, 0), \bm^{(2)}=(1,0, \ldots, 0), \ldots, \bm^{(d+1)}=(0, \ldots, 0,1).$ For $d$-dimensional multi-index $\bm=\left(m_1, \ldots, m_d\right)$, where $m_{j} \geq 0$ are integers, we define the absolute value $|\bm|=m_{1}+\ldots+m_{d}$, the factorial $\bm !=m_1 ! \ldots m_d !$, the power $\bu^{\bm}=u_1^{m_1} \ldots u_d^{m_d}$ and the differentiation operator $D^{\bm}=\frac{\partial^{|\bm|}}{\partial u_1^{m_1} \ldots \partial u_{d}^{m_d}}$. We denote by $\proj_{\com}(\bx)$ the Euclidean projection of $\bx\in \mathbb{R}^{d}$ onto $\com$.

\section{Definitions and assumptions}
\label{sec:definitions}

We first introduce the class of $\beta$-H{\"o}lder functions that will be used throughout the paper.
For $\beta,L>0$,  we denote by $\mathcal{F}_{\beta}(L)$ the  class of $\ell=\lfloor\beta\rfloor$ times differentiable functions $f: \mathbb{R}^d \rightarrow \mathbb{R}$ such that
\[
\left|f(\bx)-\sum_{|\bm| \leq l} \frac{1}{\bm !} D^{\bm} f(\bx)(\bx-\bx^{\prime})^{\bm} \right| \leq L\|\bx-\bx^{\prime}\|^{\beta}, \quad \forall \bx, \bx^{\prime} \in \mathbb{R}^{d}.
\]

Our estimators will be based on kernels satisfying the following assumption. 

\begin{assumption}\label{ass:kernel}
The kernel $K:\mathbb{R}^d\to \mathbb{R}$ has a compact support $\Supp(K)$ contained in the unit Euclidean ball, and satisfies the conditions
\begin{align*}
    K(\bu)\geq 0, \quad \quad \int K(\bu)\drm\bu = 1, \quad\quad \sup_{\bu \in \mathbb{R}^d}K(\bu) < \infty\enspace.
\end{align*}
Furthermore, we assume that $K$ is a $L_K$-Lipschitz function, that is, for any $\bx,\by \in \mathbb{R}^d$ we have
\begin{align*}
    |K(\bx) - K(\by)| \leq L_K\norm{\bx-\by}.
\end{align*}
\end{assumption}

\begin{assumption}\label{assnoise}
It holds for all $i,i'\in [n]$, that: (i) $\xi_i$ and $\bx_{i'}$ are independent; (ii) $\Exp[\xi_i] = 0$; (iii) there exists $\sigma>0$ such that  $\Exp\left[\exp(\xi_i/\sigma)\right] < \infty$ (sub-exponential assumption).
\end{assumption}
\begin{assumption}\label{mainassump} We consider model \eqref{mainmodel} with $f:\mathbb{R}^d \rightarrow \mathbb{R}$ satisfying the following assumptions.
\begin{enumerate}
\itemsep0em
\item[(i)] The function $f$ attains its minimum on $\com$ at point $\bx^*$. 
\item[(ii)] The function $f$ belongs to H{\"o}lder functional class $\mathcal{F}_{\beta}(L)$ with $\beta\geq 2.$
\item[(iii)] There exists $\alpha>0$ such that the function $f$ is $\alpha$-strongly convex on $\com$ i.e. for any $\bx,\by\in\com$, it satisfies $$f(\by) \geq f(\bx) + \langle\nabla f(\bx), \by-\bx\rangle + \frac{\alpha}{2}\norm{\bx - \by}^2\enspace.$$
\item[(iv)] The function $f$ is uniformly bounded on the set $\com' = \{\bx + \by: \bx\in \com\quad \text{and}\quad \Vert \by\Vert \le 1 \}$, that is, $\sup_{\bx \in \com'}|f(\bx)| \leq M$, where $M$ is a constant.
\end{enumerate}
\end{assumption}

We denote by $\mathcal{F}_{\beta,\alpha}(L)$ the class of regression functions $f$ satisfying Assumption \ref{mainassump}. 

Next, we introduce an assumption on the  distribution
of $\bx_i$'s. 
\begin{assumption}\label{distass}
The random vectors $\bx_1,\dots,\bx_n$ are i.i.d. with distribution admitting a density $p(\cdot)$ with respect to the Lebesgue measure such that
$$
0<p_{\min} \leq p(\bx) \leq p_{\max}<\infty, \quad \forall \bx \in \com'.
$$
\end{assumption}
{Throughout this paper, we use the notation $\cst$, $\cst_i$, $i=0,1,\dots$ for positive constants that can only depend on $d$, $\com$, $\beta$, $L$, $M$, $p_{\max}$, $p_{\min}$, $K$, and $\sigma$, where the dependence on $d$ is at most of polynomial order with the degree of the polynomial only depending on $\beta$. The values of $\cst$ and $\cst_i$ can vary from line to line. We note that the dependence on the strong convexity parameter $\alpha$ is not included in the constants $\cst$ and $\cst_i$ since we explicitly specify it in the upper and lower bounds.}

\section{Estimating the minimizer}

We estimate the minimizer $\bx^*$ via an  approximation of the gradient algorithm, where we replace the gradient $\nabla f(\bz)$ by its local polynomial estimator. 
 The objective function $f \in \mathcal{F}_{\beta}(L)$ in model \eqref{mainmodel} can be well approximated by its Taylor polynomial of order $\ell$ in the neighbourhood of the target point $\bz$,
\[
f(\bx) \approx \sum_{|\bm| \leq \ell} \frac{1}{\bm !} D^{\bm} f(\bz)(\bz-\bx)^{\bm} = \btheta^{\top}(\bz)\bU\left(\frac{\bx-\bz}{h}\right)\enspace,
\]
where $\bx$ is sufficiently close to $\bz$ and, for $h>0$, 

\begin{align}\label{co-vector}
\bU(\bu) &= \left(\frac{\bu^{\bm^{(1)}}}{\bm^{(1)}!},\dots, \frac{\bu^{\bm^{(S)}}}{\bm^{(S)}!}\right)^{\top}, \
\btheta(\bz) =\left(h^{|\bm^{(1)}|}D^{\bm^{(1)}}f(\bz),\dots,h^{|\bm^{(S)}|}D^{\bm^{(S)}}f(\bz) \right)^{\top}.
\end{align}
The {\em local polynomial estimator} of $\btheta(\bz)$ (see, e.g., \cite{stone1980}, \citep[Section 1.6]{Tsybakov09}) is defined as follows:
\[
\hat{\btheta}_k(\bz)\in \argmin_{\btheta \in \mathbb{R}^{S}} \sum_{i=1}^{k}\left[y_i-\btheta^{\top} \bU\left(\frac{\bx_{i}-\bz}{h}\right)\right]^{2} K\left(\frac{\bx_i-\bz}{h}\right)\enspace,
\]
where $K:\mathbb{R}^d\to \mathbb{R}$ is a kernel satisfying Assumption  \ref{ass:kernel}. 
Let the matrix $B_k(\bz)$ and the vector $\bD_k(\bz)$ be defined as
$$
\begin{gathered}
B_k(\bz)=\frac{1}{k h^d} \sum_{i=1}^{k} \bU\left(\frac{\bx_i - \bz}{h}\right)\bU^{\top}\left(\frac{\bx_i - \bz}{h}\right)K\left(\frac{\bx_i - \bz}{h}\right)\enspace, \\
\bD_k(\bz)=\frac{1}{k h^d} \sum_{i=1}^{k} y_{i} \bU\left(\frac{\bx_{i}-\bz}{h}\right) K\left(\frac{\bx_{i}-\bz}{h}\right).
\end{gathered}
$$
If the matrix $B_k(\bz)$ is invertible we have  
\[
\hat{\btheta}_k(\bz) = B_k(\bz)^{-1}\bD_k(\bz)
\] 
and an estimator for $\nabla f(\bz)$ can be defined in the form
\begin{equation}\label{estimatedgradient}
    \bg_k(\bz) = 
    \frac{1}{h} A \hat{\btheta}_k(\bz),
\end{equation}
where $A$ is the matrix with elements 
\begin{align*}
    A_{i,j} = \begin{dcases}
     1, &\text{if } \quad j = i+1\\
    0, &\text{otherwise},
    \end{dcases}
\end{align*}
for $i\in [d],$ and $j\in[S]$. 

Since $B_k(\bz)$ is not necessarily invertible, instead of using 
the estimator \eqref{estimatedgradient} we consider its regularized version. Namely, we add a regularization constant $\lambda>0$ to the diagonal entries of $B_k(\bz)$ and define $B_{k,\lambda}(\bz) = B_{k}(\bz) + \lambda \mathbf{I}$, where $\mathbf{I}$ is the identity matrix.  This leads to the following regularized estimator of the gradient:
\begin{equation}\label{estimatedgradientreg}
\bg_{k,\lambda}(\bz) = \frac{1}{h} A \hat{\btheta}_{k,\lambda}(\bz) := \frac{1}{h} A (B_k(\bz)+\lambda \mathbf{I})^{-1}\bD_k(\bz).
\end{equation}
The corresponding approximate gradient descent procedure is presented as Algorithm \ref{algo}. It outputs  $\bz_k$ that will be used as an estimator of $\bx^*$.
At round $k$ of Algorithm \ref{algo}, the matrix $B_{k,\lambda}(\bz_k) = B_{k}(\bz_k) + \lambda \mathbf{I} $ and the vector $\bD_k(\bz_k)$ can be computed recursively based on the first $k$ observations. {An advantage of this method is its recursive construction enabling Algorithm \ref{algo} to handle incoming i.i.d. data points $(\bx_k,y_k)$ that are received sequentially in an online fashion.} 

\begin{algorithm}[t!]
\caption{Passive Zero-Order Projected Gradient {Descent}}\label{algo}
\begin{algorithmic}
\State{\bfseries Requires} ~ Kernel $K :\mathbb{R}^d\rightarrow \mathbb{R}$, step sizes $\eta_k>0$, parameters $h_k =  \left(\frac{\log(k+1)}{k}\right)^{\frac{1}{2\beta+d}}$ and $\lambda_k =  \left(\frac{\log(k+1)}{k}\right)^{\frac{\beta}{2\beta+d}},$ for $k\in[n]$.
\vspace{.1cm}
\State {\bfseries Initialization} ~ Choose $\bz_1\in \com$, and set {$\eta_k = \frac{2}{\alpha k}$}, for $k\in[n]$.
\vspace{.1cm}
\State {\bfseries For}  $k\in[n]$
\vspace{.1cm}
\State \qquad {1.} ~~~Compute $\bg_{k,\lambda_k}(\bz_k) = h_{k}^{-1}\left(A{B_{k,\lambda_k}(\bz_k)}^{-1}\bD_k(\bz_k)\right).$
\vspace{.05cm}
\State \qquad 2. ~~~Update $\bz_{k+1} =\proj_{\com}\left(\bz_k - \eta_k\bg_{k,\lambda_k}(\bz_k)\right).$
\vspace{.05cm}
\State {\bfseries Return} ~ $(\bz_k)_{k=1}^n$
\end{algorithmic}
\end{algorithm}



The following theorem gives a bound on the estimation error of Algorithm \ref{algo}.

\begin{theorem}\label{pointest}
Let Assumptions \ref{ass:kernel} -- \ref{distass} hold and $\nabla f(\bx^*)=0$. Then, for $\bz_n$ generated by Algorithm \ref{algo} we have
\begin{equation}\label{main_bound_for_minimizer}
        \Exp[\norm{\bz_n - \bx^*}^2] \leq \cst\min \left(1, \left(\frac{\log(n)}{n}\right)^{\frac{2(\beta-1)}{2\beta+d}}\alpha^{-2}\right)\enspace.
\end{equation}
\end{theorem}
The proof of Theorem \ref{pointest}
uses the strong convexity of $f$ to obtain an upper bound for the conditional expectation $\Exp[\|\bz_{k+1}-\bx^*\|^2|\bz_k]$, which depends on the bias term $\norm{\Exp[\bg_{k,\lambda_k}(\bz_k)|\bz_k] - \nabla f(\bz_k)}$ and on the ``variance" term $\Exp\left[\norm{\bg_{k,\lambda_k}(\bz_k)}^2|\bz_k\right]$.  These two terms need to be controlled uniformly over $\bz_k \in \Theta$, which is achieved using Lemma \ref{errorgrad}. The uniformity is the reason why the bound \eqref{main_bound_for_minimizer} includes an extra logarithmic factor compared to the optimal rate $n^{-(\beta-1)/(2\beta+d)}$ derived in \cite{tsybakov1990passive}.

One may consider the logarithmic factor appearing in  \eqref{main_bound_for_minimizer} as a price to pay for the fact that our algorithm is adaptive to the marginal density of $\bx_i$'s and is realized in online mode.   Indeed, \cite{tsybakov1990passive} considered estimators that depend on the marginal density of $\bx_i$'s and achieve the rate $n^{-(\beta-1)/(2\beta+d)}$, while Algorithm \ref{algo} is free of such dependence. On the other hand, we believe that the extra logarithmic factor in the rate can be
avoided for the estimator of $\bx^*$ defined as a minimizer over $\com$ of the local polynomial estimator of function $f$. Such a method does not need the knowledge of the marginal density but it needs the whole sample and cannot be realized in online mode. Moreover, it is computationally intractable since it requires minimization of the estimator, which is a function of general form. It remains an open question whether there exists an algorithm combining all the three advantages, that is, online realization,   adaptivity  to the marginal density and convergence with the rate $n^{-(\beta-1)/(2\beta+d)}$ with no extra logarithmic factor.

In the following theorem, we provide a bound on the optimization risk 
$\Exp\left[f(\bar{\bz}_n) - f^*\right]$,
where $\bar{\bz}_n$ is the average of the outputs of Algorithm \ref{algo} over $n$ iterations. 
\begin{theorem}\label{thm:opt-error}
Let Assumptions \ref{ass:kernel} -- \ref{distass} hold. Then, for $\bar{\bz}_n = \frac{1}{n}\sum_{k=1}^{n}\bz_k$, where $\bz_k$'s are generated by Algorithm \ref{algo}, we have 
\begin{align}\label{6}
    \Exp\left[f(\bar{\bz}_n) - f^*\right] \leq \cst\min \bigg(1, \left(\frac{\log(n)}{n}\right)^{ \frac{2(\beta-1)}{2\beta+d}}\alpha^{-1}\bigg).
\end{align}
\end{theorem}
Note that inequality \eqref{lower_opt_error} below gives a minimax lower bound for the optimization risk with the  rate $\alpha^{-1}n^{-2(\beta-1)/(2\beta+d)}$ for all {$\alpha \geq n^{-(\beta+2)/(2\beta+d)}$.} 
This fact and Theorem \ref{thm:opt-error} imply that, in the zone {$\alpha \geq n^{-(\beta+2)/(2\beta+d)}$}, the estimator $\bar{\bz}_n$ achieves the minimax optimal rate (up to a logarithmic factor) with respect to both $n$ and $\alpha$. 
This is the reason to use $\bar{\bz}_n$ instead of the last output $\bz_n$ since for $\bz_n$ we get a guarantee that scales optimally only in $n$ but not in $\alpha$. Indeed, using Theorem~\ref{pointest} and the inequality
$f({\bz}_n) - f^*\le C\norm{{\bz}_n-\bx^*}^2$, where $C>0$ is a constant,
we get the rate $\left(\frac{\log(n)}{n}\right)^{ \frac{2(\beta-1)}{2\beta+d}}\alpha^{-2}$, and not $\left(\frac{\log(n)}{n}\right)^{ \frac{2(\beta-1)}{2\beta+d}}\alpha^{-1}$ for the optimization risk. 

 For very small $\alpha$ such that {$\alpha <  n^{-(\beta+2)/(2\beta+d)}$} the lower bound \eqref{lower_opt_error} scales differently from \eqref{6}  and thus we cannot claim minimax rate optimality for the estimator $\bar{\bz}_n$ for all $n$ and $\alpha$. 
Nevertheless, for any fixed $\alpha>0$ independent of $n$, Theorem \ref{thm:opt-error} and the lower bound \eqref{lower_opt_error} imply that the estimator $\bar{\bz}_n$ attains the minimax optimal rate considered as a function only of $n$.

It is interesting to compare the result of Theorem \ref{thm:opt-error} with the optimal rates for the optimization risk in the case of active design. As discussed in the Introduction, under the active design for the same class of functions $f$ as in Theorem \ref{thm:opt-error}, the dimension $d$ disappears from the optimal rate, which upgrades to $n^{-(\beta-1)/\beta}$, cf. \cite{PT90,akhavan2020}. On the other hand, under active design and the class of $\beta$-H{\"o}lder functions (with no strong convexity) the optimal rate for the optimization risk deteriorates substantially and becomes $(n/\log n)^{-\beta/(2\beta+d)}$, cf. \cite{wang2018optimization}. For all $\beta> 2$, this is worse than the rate under passive design and strong convexity obtained in Theorem \ref{thm:opt-error}. 

\section{Estimating the minimum value of a regression function}

In this section, we apply the above results to estimate the minimum value $f^* = \min_{\bx \in \Theta}f(\bx)$ of function $f$ that belongs to the class $\mathcal{F}_{\beta, \alpha}(L).$ Note that $f(\bar{\bz}_n)$, which is analyzed in Theorem \ref{thm:opt-error} is not an estimator for $f^*$, because it depends on the unknown $f$. 

We propose a method of estimating $f^*$ that uses splitting of the data in two subsamples of equal size. Throughout this section, we assume that $n$ is an even positive integer and we set $m = n/2$. We split the data into two subsamples $D_1=\{(\bx_1,y_1),\dots,(\bx_m,y_m)\}$ and $D_2=\{(\bx_{m+1},y_{m+1})\dots,(\bx_n,y_n)\}$.  Then we apply Algorithm \ref{algo} with $D_1$ as an input to construct $\bar{\bz}_{m} = \frac{1}{m}\sum_{k=1}^{m}\bz_k$, where $\bz_k$ is the update of Algorithm \ref{algo} at round $k\in[m]$. 
Next, based on the subsample $D_2$, we construct a nonparametric estimator $\tilde f_n(\cdot)$ of $f(\cdot)$. At this step, we can use as $\tilde f_n(\cdot)$ any estimator of $f(\cdot)$, which is pointwisely rate optimal. Finally, we take $\tilde f_n(\bar{\bz}_n)$ as an estimator for $f^*$.  

To be specific, we consider as $\tilde f_n(\cdot)$ a regularized local polynomial estimator defined in the same spirit as the estimator of the gradient \eqref{estimatedgradient}. For $\lambda_{m:n}, h_{m:n}>0$, and $\bz\in \mathbb{R}^d$, we define $\hat{\btheta}_{m:n}(\bz) = \left(B_{m:n}(\bz) + \lambda_{m:n} \mathbf{I}\right)^{-1}\bD_{m:n}(\bz)$, where 
\begin{align}
    B_{m:n}(\bz)&=\frac{2}{n h_{m:n}^{d}} \sum_{i=m+1}^{n} \bU\left(\frac{\bx_i - \bz}{h_{m:n}}\right)\bU^{\top}\left(\frac{\bx_i - \bz}{h_{m:n}}\right)K\left(\frac{\bx_i - \bz}{h_{m:n}}\right)\enspace, \\
\bD_{m:n}(\bz)&=\frac{2}{n h_{m:n}^{d}} \sum_{i=m+1}^{n} y_{i} \bU\left(\frac{\bx_{i}-\bz}{h_{m:n}}\right) K\left(\frac{\bx_{i}-\bz}{h_{m:n}}\right)\enspace.
\end{align}
 The {\em regularized local polynomial estimator} of function $f$ at point $\bz$ defined as
\begin{equation}\label{pointestimation}
   \tilde f_n(\bz) = \bU^{\top}(0) \hat{\btheta}_{m:n}(\bz).
\end{equation}
The minimum value estimator that we propose is  outlined in Algorithm \ref{algo2}. 

\begin{algorithm}[t!]
\caption{Estimating the Minimum Value}\label{algo2}
\begin{algorithmic}
\State 
\State {\bfseries Requires} ~ Algorithm \ref{algo}, kernel $K :\mathbb{R}^d\rightarrow \mathbb{R}$, parameters $h_{m:n} =  n^{-\frac{1}{2\beta+d}}$, $\lambda_{m:n} = n^{-\frac{\beta}{2\beta+d}}$.
\vspace{.1cm}
\\
1. Split the data $D$ in two equal parts $D_1$ and $D_2$.
\vspace{.1cm}
\\ 
2. Use Algorithm \ref{algo} to compute $\bar{\bz}_m$ based on the subsample $D_1$.
\vspace{.1cm}
\\
3. Based on the second subsample $D_2,$ compute the estimator $\tilde f_n(\bar{\bz}_m) = \bU^{\top}(0) \hat{\btheta}_{m:n}(\bar{\bz}_m)$.
\State {\bfseries Return} ~ $\tilde f_n(\bar{\bz}_m)$
\end{algorithmic}
\end{algorithm}

It is known from \citep{stone1980} that the local polynomial estimator with no regularization, i.e., with $\lambda_{m:n} =0$, and with $h_{m:n}$ as in  Algorithm \ref{algo2} attains the minimax optimal rate of estimating the value of $f$ at a fixed point in asymptotics as $n\to\infty$. 
But for finite $n$ this estimator is not necessarily well-defined. In the next theorem, we show that its properly regularized version attains the minimax pointwise rate for all sample sizes.
\begin{restatable}{theorem}{funcest}\label{funcest}
Let Assumptions \ref{ass:kernel} -- \ref{distass} hold, and let $\tilde f_{n}$ be the estimator \eqref{pointestimation} with parameters $h_{m:n}$ and $\lambda_{m:n}$ as in Algorithm \ref{algo2}. Then 
\begin{align*}
 \sup_{\bx \in \com}   \Exp\left[\left( \tilde f_{n}(\bx) - f(\bx)\right)^{2}\right]\leq \cst n^{-\frac{2\beta}{2\beta+d}}\enspace.
\end{align*}
\end{restatable}
Inspection of the proofs  shows that the strong convexity assumption (Assumption \ref{mainassump} (iii)) is not needed in Theorem \ref{funcest}. We do not state it in the theorem since it has no incidence on the other results of the paper.

   Using Theorems \ref{thm:opt-error} and \ref{funcest} we obtain the following bounds on the convergence rate of the estimator $\hat f_n = \tilde f_n(\bar{\bz}_m)$ of the minimum value $f^*=f(\bx^*)$.
\begin{theorem}\label{minimum_value_estimation}
Let Assumptions \ref{ass:kernel} -- \ref{distass} hold and let $\hat f_n = \tilde f_n(\bar{\bz}_m)$ be the output of Algorithm \ref{algo2}. Then
\begin{equation}\label{fstarestimator}
    \Exp\left|\hat f_n - f(\bx^*)\right| \leq \cst\max\left(1,\alpha^{-1}\right)\cdot\begin{dcases}
     (\log(n)/n)^{\frac{2}{4+d}} &\text{if } \beta =2\enspace,\\
     n^{-\frac{\beta}{2\beta+d}} &\text{if } \beta > 2\enspace,
     \end{dcases}
\end{equation}
Furthermore, for $\alpha\geq n^{-\frac{\beta-2}{2\beta+d}}\log(n)^{\frac{2(\beta-1)}{2\beta+d}}$ we have 
\begin{align}\label{fstarestimator1}
    \Exp\left|\hat f_n - f(\bx^*)\right| \leq \cst n^{-\frac{\beta}{2\beta+d}}.
\end{align}
\end{theorem}
\begin{proof} Using conditioning on $\bar{\bz}_m$ and the Cauchy-Schwarz inequality we get 
\begin{align*}
    \Exp\left|\hat f_n - f(\bx^*)\right| &\leq \Exp|\tilde f_n(\bar{\bz}_m) - f(\bar{\bz}_m)| + \Exp|f(\bar{\bz}_m) - f(\bx^{*})|
    \\
    &= \Exp\left[\Exp\big[|\tilde f_n(\bar{\bz}_m) - f(\bar{\bz}_m)| \big|\bar{\bz}_m\big]\right] + \Exp|f(\bar{\bz}_m) - f(\bx^{*})|
    \\&
    \leq \Exp\left[\left(\Exp\left[\left(\tilde f_n(\bar{\bz}_m) - f(\bar{\bz}_m)\right)^{2}|\bar{\bz}_m\right]\right)^{\frac{1}{2}}\right] + \Exp|f(\bar{\bz}_m) - f(\bx^{*})| .
\end{align*}
Recalling that $\tilde f_n(\cdot)$ and $\bar{\bz}_m$ are independent and applying Theorems \ref{thm:opt-error} and \ref{funcest} we deduce that 
\begin{align}\label{eq:th:fstarestimator}
    \Exp\left|\hat f_n - f(\bx^*)\right| &\leq \cst\bigg(n^{-\frac{\beta}{2\beta+d}} + {\alpha^{-1}}\left(\frac{\log(n)}{n}\right)^{\frac{2(\beta-1)}{2\beta+d}}\bigg),
\end{align}
which implies the bound \eqref{fstarestimator1} for $\alpha\geq n^{-\frac{\beta-2}{2\beta+d}}\log(n)^{\frac{2(\beta-1)}{2\beta+d}}$, and the bound \eqref{fstarestimator}. 
\end{proof}
Theorem \ref{minimum_value_estimation} shows that estimation of $f^*=f(\bx^*)$ for smooth and strongly convex functions under passive design is realized with the same rate as function estimation at a fixed point. {If the condition $\alpha \geq n^{-\frac{\beta-2}{2\beta+d}}\log(n)^{\frac{2(\beta-1)}{2\beta+d}}$ is satisfied Theorem \ref{minimum_value_estimation} and the lower bound \eqref{lower_minimum_value} imply that the estimator~$\hat f_n$ attains the minimax optimal rate as function of both parameters $n$ and $\alpha$ on the class $\mathcal{F}_{\beta,\alpha}(L)$. For $\beta=2$ the above condition requires $\alpha$ to be big, namely, $\alpha \geq \log(n)^{\frac{2}{2\beta+d}}$. Nevertheless, for $\beta=2$ and $\alpha\ge 1$ the bounds \eqref{fstarestimator} and \eqref{lower_minimum_value} imply that $\hat f_n$ is minimax optimal up to a logarithmic factor. Regarding the regime of small $\alpha$, where $\alpha < n^{-\frac{\beta-2}{2\beta+d}}\log(n)^{\frac{2(\beta-1)}{2\beta+d}}$, we cannot claim that the optimality is achieved by $\hat f_n$ with respect to both $n$ and  $\alpha$. Indeed, in this regime, the lower bound \eqref{lower_minimum_value} does not depend on $\alpha$ while the upper bound \eqref{eq:th:fstarestimator} does. On the other hand, considering $\alpha>0$ as fixed, the dependence of the convergence rate of $\hat f_n$ on the sample size $n$ cannot be improved in a minimax sense in all regimes (up to a logarithmic factor if $\beta=2$) and it corresponds to the optimal rate of estimating a smooth function at a {fixed} point.}

In the next section we show that the rate $(n/ \log n)^{-\beta/(2\beta+d)}$ is minimax optimal for estimating the minimum value $f^*$ on the class of $\beta$-smooth regression functions {\it without} the strong convexity assumption. It coincides with the rate of function estimation in the supremum norm. Thus, the strong convexity allows us to reduce the global function reconstruction problem to a simpler, fixed point estimation leading to the rates without an extra logarithmic factor. Note also that $f^*$ cannot be estimated better than with the rate $n^{-\beta/(2\beta+d)}$  even in the oracle setting where the unknown minimizer $\bx^*$ is revealed. Indeed, in this case we still need to estimate the value of function $f$ at point $\bx^*$.

Notice that, for $\beta>2$, the convergence rate of Algorithm \ref{algo} used at the first stage to estimate the minimizer $\bx^*$ is more than needed to obtain \eqref{fstarestimator}. To achieve the minimax optimal rate for $f^*$ it suffices to estimate $\bx^*$ at a slower rate, namely, $n^{-\beta/(2\beta+d)}$ for the optimization risk. Therefore, it is not necessary to take the estimator  $\bar \bz_m$ at the first stage. It can be replaced by some suboptimal estimators. This could be beneficial since suboptimal algorithms may be computationally less costly. 

Finally, observe that  much faster rate can be obtained in the active design setting, see Table~\ref{table:comparisons}.  Specifically, $f^*$ can be estimated with the parametric rate $C n^{-1/2}$ where $C>0$ is a constant independent of the dimension $d$ and smoothness $\beta$ for any $\beta>2$ and all $n$ large enough \citep{akhavan2020}. Clearly, the rate $n^{-1/2}$ cannot be improved even  in the oracle setting where the unknown
minimizer $\bx^*$ is revealed and one  makes all queries at point $\bx^*.$

\section{Lower bounds}

The following theorem provides lower bounds for the minimax risks of arbitrary estimators
on the class $\mathcal{F}_{\beta, \alpha}(L)$. Let $w(\cdot)$ 
be a monotone non-decreasing function on $[0,\infty)$ such that $w(0)=0$ and $w\not \equiv 0$.
\begin{theorem}\label{lower_bound}
Let  $\bx_1,\dots,\bx_n$ be i.i.d. random vectors with a bounded Lebesgue density on ${\mathbb R}^d$.
Assume that the random variables $\xi_{i}$ are i.i.d. having a density $p_{\xi}(\cdot)$ with respect to the Lebesgue measure on $\mathbb{R}$ such that
\begin{equation}\label{noice_assumption}
    \exists I_{*}>0, v_{0}>0: \quad \int \left(\sqrt{p_{\xi}(u)} - \sqrt{p_{\xi}(u+v)} \right)^2  \drm u \leq I_{*} v^{2}\enspace,
\end{equation}
for $|v| \leq v_{0}$, and $(\xi_{1},\dots,\xi_{n})$ is independent of $(\bx_1,\dots,\bx_n)$. Then, for any $\alpha,L>0$, $\beta\ge 2$, we have
\begin{equation}\label{lower_minimizer}
    \inf_{\hat\bx_{n}} \sup _{f \in \mathcal{F}_{\beta, \alpha}(L)} \mathbf{E}_{f}w\left({\max\left(
\alpha n^{\frac{\beta-1}{2 \beta+d}}, n^{\frac{1}{2\beta+d}}\right)}\left\|\hat\bx_{n}-\bx^*\right\|\right) \geq c_{1},
\end{equation}
and
\begin{equation}\label{lower_minimum_value}
    \inf_{T_{n}} \sup _{f \in \mathcal{F}_{\beta, \alpha}(L)} \mathbf{E}_{f}w(n^{\frac{\beta}{2 \beta+d}}\left|T_{n}-f^*\right|) \geq c_{1}',
\end{equation}
where $\inf_{\hat\bx_{n}}$ and $\inf_{T_{n}}$ denote the infimum over all estimators of the minimizer and over all estimators of the minimum value of $f$, respectively, and $c_1>0, c_1'>0$ are constants that depend only on {$\beta, L, d, \com, I_{*}, v_{0}$, and $w(\cdot)$}. Furthemore,
\begin{equation}\label{lower_opt_error}
    \inf_{\hat\bx_{n}} \sup _{f \in \mathcal{F}_{\beta, \alpha}(L)} \mathbf{E}_{f}[ f(\hat\bx_{n})-f^*] \geq c_{2}
    \min\Big(
\alpha^{-1} n^{-\frac{2(\beta-1)}{2 \beta+d}}, \alpha n^{-\frac{2}{2\beta+d}}\Big),
\end{equation}
where $c_2>0$ is a constant that depends only on $\beta, L, d, \com, I_{*}, v_{0}$.
\end{theorem}
Condition \eqref{noice_assumption} is rather general. It is satisfied, for example, for the Gaussian distribution and also for a large class of regular densities, cf. \cite{IbrHasm-book}. 
The instance of the lower bound \eqref{lower_minimizer} with $\alpha\asymp 1$ was proved in 
\cite{tsybakov1990passive} under a more restrictive condition on the density $p_\xi$. In \eqref{lower_minimizer}, we extend this result by deriving the dependence of the lower bound on the parameter~$\alpha$, which reveals two non-asymptotic regimes. This is also reflected in \eqref{lower_opt_error}, which is a direct consequence of \eqref{lower_minimizer}.

The proof of Theorem \ref{lower_bound} is given in Section \ref{proof_lower_bound}.
It is based on a reduction  to the problem of testing two hypotheses for \eqref{lower_minimizer} and two fuzzy hypotheses  for \eqref{lower_minimum_value} (cf. \cite{Tsybakov09}). 

Combining the bound \eqref{lower_minimizer}  with   $w(u)=u^2$  and Theorem \ref{pointest} we conclude that  the estimator $\bz_n$ achieves the minimax optimal rate up to a logarithmic factor for estimating $\bx^*$ under the squared risk on the class of functions $\mathcal{F}_{\beta, \alpha}(L)$ if $\alpha\ge n^{-(\beta-2)/(2\beta+d)}$. Similar conclusion holds true for the optimization risk, see the comment after Theorem \ref{thm:opt-error}, and for the estimation of the minimum value $f^*$, see the comment after Theorem \ref{minimum_value_estimation}.  

It is interesting to compare our results about estimation of $f^*$ with the setting where $f$ is not strongly convex. To this end, we provide a minimax lower bound on estimation of $f^*$ over the class of $\beta$-H{\"o}lder functions $\mathcal{F}_{\beta}(L)$. 

\begin{theorem}\label{th:lower_bound_nonconvex}
Let  $\bx_1,\dots,\bx_n$ be i.i.d. random vectors with a bounded Lebesgue density on ${\mathbb R}^d$. Let $\xi_{i}$'s be i.i.d. Gaussian random variables with zero mean and variance $\sigma^2$ and let $(\xi_{1},\dots,\xi_{n})$ be independent of $(\bx_1,\dots,\bx_n)$. Assume that $\com$ contains an open subset of ${\mathbb R}^d$. Then, for any $\beta > 0, L>0,$ we have
\begin{equation}\label{lower_minimum_value_nonconvex}
    \inf_{T_{n}} \sup _{f \in \mathcal{F}_{\beta}(L)} \mathbf{E}_{f}w\Big(\left(\frac{n}{\log n}\right)^{\frac{\beta}{2 \beta+d}}\left|T_{n}-f^*\right|\Big) \geq c_{3},
\end{equation}
where $\inf_{T_{n}}$ denotes the infimum over all estimators of the minimum value of $f$ and $c_3>0$ is a constant that depends only on $\beta, \alpha, d, L, \com, \sigma^2$, and $w(\cdot)$.
\end{theorem}

Theorem \ref{th:lower_bound_nonconvex} implies that $\left(\frac{n}{\log n}\right)^{-\frac{\beta}{2 \beta+d}}$ is the minimax rate of estimating the minimum value $f^*$ on the class $\mathcal{F}_{\beta}(L)$. Indeed, the matching upper bound with the rate $\left(\frac{n}{\log n}\right)^{-\frac{\beta}{2 \beta+d}}$ is obtained in a trivial way if we estimate $f^*$ by the minimum of any rate optimal (in supremum norm) nonparametric estimator of $f$, for example, by the local polynomial estimator as in \cite{stone1982optimal}. 

Thus, if we drop the assumption of strong convexity, the minimax rate deteriorates only by a logarithmic factor. It suggests that strong convexity is not a crucial advantage in estimation of the minimum value of a function under the passive design. 

\section{Conclusion}

In this paper, we considered the problem of estimating the minimizer and the minimum value of a regression function from i.i.d data with a focus on highly smooth and strongly convex regression functions. We provide upper bounds for the proposed algorithms and minimax lower bounds for all estimators. We show that the minimax optimal rates of estimating the minimizer are the same as for estimating the gradient of the regression function while for the minimum value they are the same as for estimation of a regression function at a fixed point. To estimate the minimum value, we propose a two-stage procedure, where we first estimate the location of the minimum, and then the function value at the estimated location. We show that this two-stage procedure achieves minimax optimal rates of convergence.

 An interesting open question is to make our algorithms adaptive to the unknown smoothness $\beta$ and to the strong convexity parameter $\alpha$. Adaptation to $\beta$ needs developing a data-driven choice of the smoothing parameter $h$ and of the regularization parameter $\lambda.$  Optimal rates for estimation of the minimum value $f^*$ under adaptation to $\beta$ are presumably slower  by a logarithmic factor than the rates that we established in the setting with known $\beta$.

\section{Proofs}
In this section, we provide the proofs of Theorems \ref{pointest} -- \ref{funcest}, \ref{lower_bound} and \ref{th:lower_bound_nonconvex}. Section \ref{proof_main_location} is
devoted to the proofs of Theorems \ref{pointest} and \ref{thm:opt-error} on the upper bounds for Algorithm \ref{algo}.  Section \ref{proof_of_bias} provides the proof of Theorem \ref{funcest} on the behavior of regularized local polynomial estimator. In Sections \ref{proof_lower_bound} and \ref{sec:proof_lower_bound_nonconvex}, we prove the results on the lower bounds (Theorems \ref{lower_bound} and \ref{th:lower_bound_nonconvex}, respectively).


\subsection{Proof of Theorem \ref{pointest}}\label{proof_main_location}

We will use the following notation. For any $k \in [n]$, and $i \in [k]$, let 
\begin{align*}
    \bR_{i,k}(\bx)& = \bU\left(\frac{\bx_i - \bx}{h_k}\right)K\left(\frac{\bx_i - \bx}{h_k}\right),
\end{align*}
and 
$$
\bC_k(\bx)= \frac{1}{kh_k^{d}}\sum_{i=1}^{k}\bR_{i,k}(\bx)f(\bx_i), 
\quad 
\bG_{k}(\bx)= \frac{1}{kh_k^{d}}\sum_{i=1}^{k}\bR_{i,k}(\bx)\xi_i, \quad 
E_k(\bx) = \Exp[B_{k}(\bx)].
$$
Note that $\bD_{k}(\bx)=\bC_{k}(\bx)+\bG_{k}(\bx)$.

We first prove some preliminary lemmas. The following lemma provides an upper bound on the bias of $\bg_{k,\lambda_k}$. 
\begin{lemma}\label{bias1}
For $k \in [n]$, let $\bg_{k,\lambda_k}$ be defined by Algorithm \ref{algo}. Let Assumptions \ref{ass:kernel} -- \ref{distass} hold.
Then 
\begin{equation}\label{gradient_bound}
   \sup_{\bx\in\com} \norm{\Exp[\bg_{k,\lambda_k}(\bx)] - \nabla f(\bx)}\leq \cst
    \left(\frac{\log(k+1)}{k}\right)^{\frac{\beta-1}{2\beta+d}}.
\end{equation}
\end{lemma}
\begin{proof}
By Lemma \ref{elbound}(ii), matrix $E_k(\bx)$ is positive definite for all $\bx\in \com$.
 Fix $\bx\in \com$ and introduce the notation $\bphi_k = \bg_{k,\lambda_k}(\bx) - h_k^{-1}\left(AE_k(\bx)^{-1}B_{k}(\bx)\bc_{h_k}(f, \bx)\right)$, where $$\bc_h(f,\bx) = \left(h^{|\bm^{(1)}|}D^{\bm^{(1)}}f(\bx),\dots,h^{|\bm^{(S)}|}D^{\bm^{(S)}}f(\bx) \right)^{\top}$$ 
 for $h>0$.
 Noticing that 
$$ \Exp\left[h_k^{-1}A E_k(\bx)^{-1} B_{k}(\bx)\bc_{h_k}(f,\bx)\right]=h_k^{-1}A\bc_{h_k}(f,\bx)  =\nabla f(
    \bx)$$ 
    we get
\begin{align*}
    \Exp\left[\bphi_k\right] =  \Exp\left[\bg_{k,\lambda_k}(\bx)\right] - \nabla f(
    \bx).
\end{align*}
 Thus, to conclude the proof we need to bound $\norm{\Exp\left[\bphi_k\right]}$ from above. Define
\begin{align*}
    \bpsi_{1,k} &= h_k^{-1}\left(AE_k(\bx)^{-1}\bC_k(\bx)\right),
    \\\bpsi_{2,k}&= h_k^{-1}\left(A(E_k(\bx)+\lambda_k\mathbf{I})^{-1}\bC_k(\bx)\right).
\end{align*}
We have
\begin{align*}
    \norm{\Exp[\bphi_k]} 
    &
    \leq \underbrace{\norm{\Exp[\bpsi_{1,k} - h_k^{-1}\left(AE_k(\bx)^{-1}B_k(\bx)\bc_k(f, \bx)\right)]}}_{\text{term I}}
    +
    \underbrace{\norm{\Exp[\bpsi_{2,k} - \bpsi_{1,k}]}}_{\text{term II}} 
    \\
    &\quad\quad + \underbrace{\norm{\Exp[\bg_{k,\lambda_k}(\bx) -  \bpsi_{2,k}]}}_{\text{term III}}.
\end{align*}
We now establish upper bounds for each of the three terms in the above expression. We have
\begin{align*}
    \text{term I} &= h_k^{-1}\norm{AE_k(\bx)^{-1}\Exp\left[\frac{1}{kh_k^{d}}\sum_{i=1}^{k}\bR_{i,k}(\bx)\left(f(\bx_i) - \bU^{\top}\left(\frac{\bx_i - \bx}{h_k}\right)\bc_{h_k}(f,\bx)\right)\right]}
    \\
    &\leq h_k^{-1}\norm{AE_k(\bx)^{-1}}_{\op}\norm{\Exp\left[\frac{1}{kh_k^{d}}\sum_{i=1}^{k}\bR_{i,k}(\bx)\left(f(\bx_i) - \bU^{\top}\left(\frac{\bx_i - \bx}{h_k}\right)\bc_{h_k}(f,\bx)\right)\right]}.
\end{align*}
By Lemma \ref{elbound}(ii) and the fact that $\norm{A}_{\op}\leq 1$ we have $\norm{AE_k(\bx)^{-1}}_{\op}\leq \lambda_{\min}^{-1}$. Therefore,
\begin{align*}
    \text{term I} &\leq h_{k}^{-1}\lambda_{\min}^{-1}\left(\frac{1}{kh_k^{d}}\sum_{i=1}^{k}\Exp\left[\norm{\bR_{i,k}(\bx)\left(f(\bx_i) - \bU^{\top}\left(\frac{\bx_i - \bx}{h_k}\right)\bc_{h_k}(f,\bx)\right)}\right]\right) \enspace.
\end{align*}
Since by Assumption \ref{mainassump}(ii), $f \in \mathcal{F}_{\beta}(L)$ for any $i\in[k]$ we have
\begin{align*}
    |f(\bx_i)-\bU^{\top}\left(\frac{\bx_i - \bx}{h_k}\right)\bc_{h_k}(f,\bx)|\leq L\norm{\bx - \bx_i}^{\beta},
\end{align*}
so that
\begin{align}
\begin{aligned}\label{eq:dop1}
 \text{term I} &\leq Lh_{k}^{-1}\lambda_{\min}^{-1}\left(\frac{1}{kh_k^{d}}\sum_{i=1}^{k}\Exp\left[\norm{\bR_{i,k}(\bx)}\norm{\bx - \bx_i}^{\beta}\right]\right)
    \\&= Lh_k^{-d-1}\lambda_{\min}^{-1}\int_{\mathbb{R}^d} \norm{\bx - \bu}^{\beta}\norm{\bU\left(\frac{\bu - \bx}{h_k}\right)K\left(\frac{\bu - \bx}{h_k}\right)}p(\bu)\drm \bu
    \\&= Lh_{k}^{\beta-1}\lambda_{\min}^{-1}\int_{\mathbb{R}^d}\norm{\bw}^{\beta}\norm{U(\bw)K(\bw)}p(\bx + h_k\bw)\drm\bw
    \leq \cst h_k^{\beta-1}\enspace.   
\end{aligned} 
\end{align}
Next,
\begin{align*}
        \text{term II} &= h_k^{-1}\norm{A \left(\left(E_k(\bx)+\lambda_k\mathbf{I}\right)^{-1} - E_k(\bx)^{-1}\right)\Exp\left[\bC_k(\bx)\right]} 
    \\&\leq h_k^{-1}\lambda_k\norm{A}_{\op}\norm{E_k(\bx)^{-1}}_{\op}\norm{\left(E_k(\bx)+\lambda_k\mathbf{I}\right)^{-1}}_{\op}\Exp\left[\norm{\bC_k(\bx)}\right]\enspace.
\end{align*}
By Assumption \ref{mainassump}(iv), we have $\sup_{\bx \in \com'}\vert f(\bx)\vert \leq M$. Using this inequality and Lemma \ref{elbound}(i) we get 
\begin{align}\label{eq:dop2}
 \sup_{\bx \in \com}\Exp\left[\norm{\bC_k(\bx)}\right] \leq Mp_{\max}\nu_{1}.   
\end{align}
 Moreover, Lemma \ref{elbound}(ii) implies that $\norm{E_k(\bx)^{-1}}_{\op}\norm{\left(E_k(\bx)+\lambda_k\mathbf{I}\right)^{-1}}_{\op} \leq \lambda_{\min}^{-2}$. Therefore, 
\begin{align}\label{eq:dop3}
        \text{term II} \leq \cst h_k^{-1}\lambda_k.
\end{align}
Finally, we bound term III. Using Assumption \ref{assnoise} it is easy to see that 
\begin{align}\label{eq:dop0}
 \Exp[\bg_{k,\lambda_k}(
\bx)] = \Exp\left[h_k^{-1}A{B_{k,\lambda_k}(\bx)^{-1}}\bC_{k}(\bx)\right].
\end{align}
Taking into account \eqref{eq:dop0} and the fact that $E_k(\bx)+\lambda_k\mathbf{I}=\Exp[B_{k,\lambda_k}(\bx)]$ we obtain  
\begin{align*}
    \text{term III} &\leq  h_k^{-1}\norm{\Exp\bigg[A \left({B_{k,\lambda_k} (\bx)^{-1}}- \left(\Exp[B_{k,\lambda_k}(\bx)]\right)^{-1}\right)
    \left(\bC_k(\bx)-\Exp\left[\bC_k(\bx)\right]\right)\bigg]}
    \\&\quad\quad+ h_k^{-1}\norm{\Exp\bigg[A \left({B_{k,\lambda_k} (\bx)^{-1}} - \left(\Exp[B_{k,\lambda_k}(\bx)]\right)^{-1}\right)
    \Exp\left[\bC_k(\bx)\right]\bigg]}
    \\&\leq h_k^{-1}\Exp\left[\norm{{B_{k,\lambda_k} (\bx)^{-1}} - \left(\Exp[B_{k,\lambda_k}(\bx)]\right)^{-1}}_{\op}\norm{\bC_k(\bx)-\Exp\left[\bC_k(\bx)\right]}\right]
    \\&\quad\quad+h_k^{-1}\Exp\left[\norm{{B_{k,\lambda_k} (\bx)^{-1}} - \left(\Exp[B_{k,\lambda_k}(\bx)]\right)^{-1}}_{\op}\right]\sup_{\bx\in\com}\Exp\left[\norm{\bC_k(\bx)}\right]
    \enspace.
\end{align*}
Using Lemma \ref{covariance} and \eqref{eq:dop2} we find:
\begin{align*}
    \text{term III} \leq \cst \left(k^{-1}h_k^{-d-1} +h_k^{-1}\Exp\left[\norm{B_{k,\lambda_k}(\bx)^{-1} - \left(\Exp[B_{k,\lambda_k}(\bx)]\right)^{-1}}_{\op}\right]\right).
\end{align*}
 This and Lemma \ref{boundB} yield 
\begin{align}\label{eq:dop4}
    \text{term III}\leq
    \cst\left( k^{-1}h_k^{-d-1}+k^{-\frac{1}{2}}h_k^{-\frac{d}{2} - 1}\right) \leq \cst k^{-\frac{1}{2}}h_k^{-\frac{d}{2} - 1}\enspace.
\end{align}
Combining \eqref{eq:dop1}, \eqref{eq:dop3} and \eqref{eq:dop4} we obtain
\begin{align*}
    \norm{\Exp[\bg_{k,\lambda_k}(\bx)] - \nabla f(\bx)}\leq \cst  \left(h_{k}^{\beta-1} + h_k^{-1}\lambda_k + h_k^{-1-\frac{d}{2}}k^{-\frac{1}{2}}\right).
\end{align*}
Since $h_k = \left(\frac{\log(k+1)}{k}\right)^{\frac{1}{2\beta+d}}$ and $\lambda_k = \left(\frac{\log(k+1)}{k}\right)^{\frac{\beta}{2\beta+d}}$ this inequality implies the result of the lemma.
\end{proof}

The next lemma provides a bound on the stochastic component of the error uniformly over $\Theta.$
\begin{lemma}\label{variance1}
Let $\bg_{k,\lambda_k}$ be defined by Algorithm \ref{algo}, and let  Assumptions \ref{ass:kernel} -- \ref{distass} hold. Then
\begin{align*}
\Exp\left[\sup_{\bx\in\com}\norm{\bg_{k,\lambda_k}(\bx) - \Exp\left[\bg_{k,\lambda_k}(\bx)\right]}^2\right] \leq \cst\left(\frac{\log(k+1)}{k}\right)^{\frac{2(\beta-1)}{2\beta+d}}.
\end{align*}
\end{lemma}
\begin{proof}
Recalling that $\bD_k(\bx)=\bC_k(\bx)+\bG_k(\bx)$, where $\bC_k(\bx)= \frac{1}{kh_k^{d}}\sum_{i=1}^{k}\bR_{i,k}(\bx)f(\bx_i)$,  $\bG_k(\bx) = \frac{1}{kh_k^{d}}\sum_{i=1}^{k}\bR_{i,k}(\bx)\xi_i$,  and taking into account \eqref{eq:dop0} we obtain
\begin{align}
&  \Exp\left[\sup_{\bx\in\com}\norm{\bg_{k,\lambda_k}(\bx) - \Exp\left[\bg_{k,\lambda_k}(\bx)\right]}^2\right] 
  \leq 2h_k^{-2}\underbrace{\Exp\left[\sup_{\bx\in\com}\norm{A{B_{k,\lambda_k}(\bx)}^{-1}\bG_k(\bx)}^{2}\right]}_{\text{term I}}
  \nonumber
\\
&+2h_k^{-2}\underbrace{\Exp\left[\sup_{\bx\in\com}\norm{A{B_{k,\lambda_k}(\bx)}^{-1}\bC_k(\bx) - \Exp\left[A{B_{k,\lambda_k}(\bx)}^{-1}\bC_k(\bx) \right]}^{2}\right]}_{\text{term II}}.
\label{eq:dop7}
\end{align}
Recalling that $\norm{A}_{\op}\le 1$ and using 
the Cauchy-Schwarz inequality and Lemmas \ref{boundBsup} and \ref{boundsupsubG} we get
\begin{align}
\begin{aligned}\label{eq:dop5}
    \text{term I} &\leq \left(\Exp\left[\sup_{\bx\in\com}\norm{{B_{k,\lambda_k}(\bx)}^{-1}}_{\op}^{4}\right]\Exp\left[\sup_{\bx\in\com}\norm{\bG_k(\bx)}^4\right]\right)^{\frac{1}{2}}
    \\
    &\le  \cst k^{-1}h_k^{-d}\log(k+1).
\end{aligned}
\end{align}
Next, using again the Cauchy-Schwarz inequality we obtain
\begin{align}
\begin{aligned}\label{eq:dlinnoe}
 \text{term II}
    &\leq 3\Exp\left[\sup_{\bx\in\com}\norm{{B_{k,\lambda_k}(\bx)}^{-1}\left(\bC_k(\bx) - \Exp\left[\bC_k(\bx)\right]\right)}^2\right]
    \\
     & \quad + 3\Exp\left[\sup_{\bx\in\com}\norm{\left({B_{k,\lambda_k}(\bx)}^{-1}-\left(\Exp\left[{B_{k,\lambda_k}(\bx)}\right]\right)^{-1}\right)\Exp\left[\bC_k(\bx)\right]}^2\right]
    \\
     & \quad + 3 \sup_{\bx\in\com}\norm{\Exp\left[\left(\left(\Exp\left[{B_{k,\lambda_k}(\bx)}\right]\right)^{-1}-{B_{k,\lambda_k}(\bx)}^{-1}\right)\bC_k(\bx)\right]}^2
     \\
  &\le 3\Exp\left[\sup_{\bx\in\com}\norm{{B_{k,\lambda_k}(\bx)}^{-1}\left(\bC_k(\bx) - \Exp\left[\bC_k(\bx)\right]\right)}^2\right]
    \\
     & \quad + 3\Exp\left[\sup_{\bx\in\com}\norm{{B_{k,\lambda_k}(\bx)}^{-1}-\left(\Exp\left[{B_{k,\lambda_k}(\bx)}\right]\right)^{-1}}_{\op}^2 \sup_{\bx\in\com}\norm{\Exp\left[\bC_k(\bx)\right]}^2\right]
    \\
     & \quad + 3 \sup_{\bx\in\com} \Exp\left[
     \norm{\left(\Exp\left[{B_{k,\lambda_k}(\bx)}\right]\right)^{-1}-{B_{k,\lambda_k}(\bx)}^{-1}}_{\op}^2\right]\Exp[\norm{\bC_k(\bx)}^2]   
\\
&\leq 3\underbrace{\Exp\left[\sup_{\bx\in\com}\norm{{B_{k,\lambda_k}(\bx)}^{-1}\left(\bC_k(\bx) - \Exp\left[\bC_k(\bx)\right]\right)}^2\right]}_{\text{term III}} 
\\&\phantom{\leq}+ 6\underbrace{\Exp\left[\sup_{\bx\in\com}\norm{{B_{k,\lambda_k}(\bx)}^{-1}-\left(\Exp\left[{B_{k,\lambda_k}(\bx)}\right]\right)^{-1}}_{\op}^{2}\right]\sup_{\bx\in\com}\Exp\left[\norm{\bC_{k}(\bx)}^{2}\right]}_{{\text{term IV}}}.   
\end{aligned}  
\end{align}
The Cauchy-Schwarz inequality and Lemmas \ref{boundBsup} and \ref{boundC} imply:
\begin{align*}
\begin{aligned}
   \text{term III} &\leq \left(\Exp\left[\sup_{\bx\in\com}\norm{{B_{k,\lambda_k}(\bx)}^{-1}}^{4}_{\op}\right]\Exp\left[\sup_{\bx\in\com}\norm{\bC_k(\bx) - \Exp\left[\bC_k(\bx)\right]}^{4}\right]\right)^{\frac{1}{2}} 
   \\
   &\le \cst k^{-1}h_k^{-d}\log\left(k+1\right).
\end{aligned}
\end{align*}
Moreover, again by the Cauchy-Schwarz inequality, \eqref{eq:lem:boundCX-1} and Lemma \ref{boundBsup} we get
$$
    \text{term IV} \leq  \cst k^{-1}h_k^{-d}\log\left(k+1\right).
    $$
Thus, 
\begin{align}\label{eq:dop6}
    \text{term II}\leq \cst k^{-1}h_k^{-d}\log(k+1).
\end{align}
We conclude the proof by combining \eqref{eq:dop7} -- \eqref{eq:dop6} and using the fact that $h_k = \left(\frac{\log(k+1)}{k}\right)^{\frac{1}{2\beta+d}}$.
\end{proof}
\begin{lemma}\label{errorgrad}
Let $\bg_{k,\lambda_k}$ be defined by Algorithm \ref{algo}, and let  Assumptions \ref{ass:kernel} --
\ref{distass} hold. Then 
\begin{align}\label{eq:dop8}
    \Exp\left[\sup_{\bx\in\com}\norm{\bg_{k,\lambda_k}(\bx) - \nabla f(\bx)]}^2\right] \leq \cst \left(\frac{\log(k+1)}{k}\right)^{\frac{2(\beta-1)}{2\beta+d}},
\end{align}
and
\begin{align}\label{eq:dop9}
    \Exp\left[\sup_{\bx\in\com}\norm{\bg_{k,\lambda_k}(\bx)}^2\right] \leq \cst. 
\end{align}
\end{lemma}
\begin{proof}
   The bound \eqref{eq:dop8} is immediate in view of Lemmas \ref{bias1} and \ref{variance1}. Next, note that, since $f$ is uniformly bounded by $M$ on $\com$ and $f\in \mathcal{F}_{\beta}(L)$, the value   $ \sup_{\bx\in\com}\norm{\nabla f(\bx)}$ is bounded by a constant depending only on $M$, $L$, $d$ and $\beta$. This fact and 
 \eqref{eq:dop8} imply \eqref{eq:dop9}.  
\end{proof}

{\bf Proof of Theorem \ref{pointest}.}
By the definition of Algorithm \ref{algo} and the contracting property of the Euclidean projection, for any $k \in [n]$ we have
\begin{align*}
   {\Exp[\norm{\bz_{k+1} - \bx^*}^2|\bz_k]} \leq  \norm{\bz_k - \bx^*}^2 - \frac{4}{\alpha k}
   \left( \bz_k - \bx^*\right)^{\top}\Exp\left[\bg_{k,\lambda_k}(\bz_k)|\bz_k\right]  + \frac{4}{\alpha^2k^2}\Exp\left[\norm{\bg_{k,\lambda_k}(\bz_k)}^2|\bz_k\right].
\end{align*}
We further obtain 
\begin{align}\nonumber
\Exp[a_{k+1}|\bz_k] &\leq  a_{k} - \frac{4}{\alpha k}\left( \bz_k - \bx^*\right)^{\top}\nabla f(\bz_k)   +\frac{4}{\alpha k}\norm{\bz_k - \bx^*}\norm{\Exp[\bg_{k,\lambda_k}(\bz_k)|\bz_k] - \nabla f(\bz_k)} 
\\
\label{thm1:1}&\quad\quad+\frac{4}{\alpha^2k^2}\Exp\left[\norm{\bg_{k,\lambda_k}(\bz_k)}^2|\bz_k\right]\enspace,
  \end{align}
where $a_k = \norm{\bz_k - \bx^*}^2$. Since $f$ is an $\alpha$-strongly convex function we have
\begin{align}\label{thm1:2}
    \alpha a_k \leq \left( \bz_k - \bx^*\right)^{\top} \nabla f(\bz_k) \enspace.
\end{align}
Combining \eqref{thm1:1} and \eqref{thm1:2} yields
\begin{align}\label{thm1:3}
   \nonumber\Exp[a_{k+1}|\bz_k] &\leq \left(1 - \frac{4}{k}\right)a_k+\frac{4}{\alpha k}\norm{\bz_k - \bx^*}\norm{\Exp[\bg_{k,\lambda_k}(\bz_k)|\bz_k] - \nabla f(\bz_k)} 
   \\
   &\quad +  \frac{4}{\alpha^2k^2}\Exp\left[\norm{\bg_{k,\lambda_k}(\bz_k)}^2|\bz_k\right].
  \end{align}
 Since $2a b \leq \gamma a^2 + \frac{b^2}{\gamma}$ for any $a, b \in \mathbb{R}$ and $\gamma > 0$ we deduce, with $\gamma=3\alpha/2$, that
 \begin{align}
 \begin{aligned}\label{eq:dopoln}
  \norm{\bz_k  - \bx^*}\norm{\Exp[\bg_{k,\lambda_k}(\bz_k)|\bz_k] - \nabla f(\bz_k)} &\leq \frac{3\alpha}{4}a _k + \frac{1}{3\alpha}\norm{\Exp[\bg_{k,\lambda_k}(\bz_k)|\bz_k] - \nabla f(\bz_k)}^2
     \\
     &\leq \frac{3\alpha}{4}a _k + \frac{1}{3\alpha}\Exp[\norm{\bg_{k,\lambda_k}(\bz_k)- \nabla f(\bz_k)}^2|\bz_k].  
\end{aligned}   
 \end{align}
Using this inequality in \eqref{thm1:3} and taking the expectations yields
\begin{align*}
   \nonumber r_{k+1} &\leq \left(1 - \frac{1}{k}\right) r_k+\frac{4}{3\alpha^2 k}\Exp\left[\norm{\bg_{k,\lambda_k}(\bz_k) - \nabla f(\bz_k)}^2\right] +  \frac{4}{\alpha^2k^{2}}\Exp\left[\norm{\bg_{k,\lambda_k}(\bz_k)}^2\right]
   \\&\leq \left(1 - \frac{1}{k}\right)r_k+\frac{4}{3\alpha^2 k}\Exp\left[\sup_{\bx\in\com}\norm{\bg_{k,\lambda_k}(\bx) - \nabla f(\bx)}^2\right] +  \frac{4}{\alpha^2k^{2}}\Exp\left[\sup_{\bx\in\com}\norm{\bg_{k,\lambda_k}(\bx)}^2\right],
  \end{align*}
where $r_k = \Exp[a_{k}]$.  By invoking Lemma \ref{errorgrad}  we deduce that
\begin{align*}
 r_{k+1} &\leq \left(1 - \frac{1}{k}\right)r_{k}+\cst k^{-1 - \frac{2(\beta-1)}{2\beta+d}}\log(k+1)^{\frac{2(\beta-1)}{2\beta+d}}\alpha^{-2}\enspace.
\end{align*}
Finally, applying Lemma \ref{lemm:akhavan_recursive} with $b_k=r_k \log(k)^{-\frac{2(\beta-1)}{2\beta+d}}$ we obtain the following bound that concludes the proof:
\begin{align*}
    \Exp[\norm{\bz_n - \bx^*}^2] \leq \cst \Big(\frac{(\text{diam}(\com))^2}{n} + n^{-\frac{2(\beta-1)}{2\beta+d}}\alpha^{-2}\Big) \log(n)^{\frac{2(\beta-1)}{2\beta+d}},
\end{align*}
where $\text{diam}(\com) = \sup_{\bx,\by \in \com}\norm{\bx - \by}$ and we used the fact that $\log(n+1) \leq 2\log(n)$ for $n\geq 2$.

\bigskip

{\bf Proof of Theorem \ref{thm:opt-error}.}
The definition of Algorithm \ref{algo} implies the inequality $\norm{\bz_{k+1}-\bx^*}^{2} \leq \norm{\bz_{k}-\eta_k\bg_{k,\lambda_k}(\bz_k) - \bx^*}^2$. Therefore, 
\begin{align}\label{eq1:opt_error}
    \langle \bg_{k,\lambda_k}(\bz_k), \bz_k - \bx^*\rangle \leq \frac{\norm{\bz_k-\bx^*}^2-\norm{\bz_{k+1}-\bx^*}^2}{2\eta_k}+\frac{\eta_k}{2}\norm{\bg_{k,\lambda_k}(\bz_k)}^2\enspace.
\end{align}
On the other hand, by Assumption \ref{mainassump}(iii) we have
\begin{align}\label{eq2:opt_error}
    f(\bz_k) - f^*\leq \left( \bz_k - \bx^*\right)^{\top}\nabla f(\bz_k)-\frac{\alpha}{2}\norm{\bz_k - \bx^*}^2\enspace. 
\end{align}
Combining \eqref{eq1:opt_error} and \eqref{eq2:opt_error} gives
\begin{align*}
    \Exp\left[f(\bz_k) - f^*|\bz_k\right]&\leq \norm{\Exp\left[\bg_{k,\lambda_k}(\bz_k)|\bz_k\right] - \nabla f(\bz_k)}
    \norm{\bz_k-\bx^*} + \frac{1}{2\eta_k}\Exp\left[a_k - a_{k+1}|\bz_k\right] \\&\phantom{\leq}+\frac{\eta_k}{2}\Exp\left[\norm{\bg_{k,\lambda_k}(\bz_k)}^2|\bz_k\right] -\frac{\alpha}{2}a_k\enspace,
\end{align*}
where $a_k = \norm{\bz_k - \bx^*}^2$. Acting as in \eqref{eq:dopoln} but now with $\gamma= \alpha/2$ we find
\begin{align*}
    \Exp\left[f(\bz_k) - f^*|\bz_k\right]
    &\leq 
    \frac{1}{\alpha}\Exp\left[\norm{\bg_{k,\lambda_k}(\bz_k) - \nabla f(\bz_k)}^2|\bz_k\right]
    + \frac{1}{2\eta_k}\Exp\left[a_k - a_{k+1}|\bz_k\right] \\&\phantom{\leq}+\frac{\eta_k}{2}\Exp\left[\norm{\bg_{k,\lambda_k}(\bz_k)}^2|\bz_k\right] -\frac{\alpha}{4}a_k .
\end{align*}
 Taking the expectations of both sides of this inequality and recalling the notation $r_k = \Exp\left[a_k\right]$ we obtain
\begin{align*}
    \Exp\left[f(\bz_k) - f^*\right]&\leq \frac{1}{\alpha}\Exp\left[\sup_{\bx\in\com}\norm{\bg_{k,\lambda_k}(\bx) - \nabla f(\bx)}^2\right]
    + \frac{1}{2\eta_k}\left(r_k - r_{k+1}\right) \\&\phantom{\leq}+\frac{\eta_k}{2}\Exp\left[\sup_{\bx\in\com}\norm{\bg_{k,\lambda_k}(\bx)}^2\right] -\frac{\alpha}{4}r_k\enspace.
\end{align*}
Lemma \ref{errorgrad} and the fact that $\eta_k = \frac{2}{\alpha k}$ further imply
\begin{align*}
    \Exp\left[f(\bz_k) - f^*\right]&\leq  \frac{\alpha k}{4}\left(r_k - r_{k+1}\right) -\frac{\alpha}{4}r_k+\cst\left(\frac{\log(k+1)}{k}\right)^{\frac{2(\beta-1)}{2\beta+d}} \alpha^{-1}\enspace.
\end{align*}
Summing both sides of this inequality from $1$ to $n$ yields
\begin{align*}
\sum_{k=1}^{n}\Exp\left[f(\bz_k) - f^*\right]&\leq  \cst n^{\frac{2+d}{2\beta+d}}\log(n)^{\frac{2(\beta-1)}{2\beta+d}} \alpha^{-1}\enspace,
\end{align*}
where we used the inequalities $\sum_{k=1}^{n}k^{-\frac{2(\beta-1)}{2\beta+d}}\leq \frac{2\beta+d}{2+d}n^{\frac{2+d}{2\beta+d}}$, and $\log(n+1)\leq 2\log(n)$ for $n\geq 2$. We conclude the proof by using the convexity of $f$ and Jensen's inequality.

\subsection{Proof of Theorem \ref{funcest}}\label{proof_of_bias}

Theorem \ref{funcest} is an immediate consequence of the bounds on the bias and variance of the regularized local polynomial estimator $f_n$ in Lemmas \ref{bias2} and \ref{var} below. In the proofs of these lemmas, we will write for brevity $h=h_{m:n}$ and $\lambda=\lambda_{m:n}$. We will use the following notation: 
\begin{align*}
    \bR_{k}(\bx)& = \bU\left(\frac{\bx_k - \bx}{h}\right)K\left(\frac{\bx_k - \bx}{h}\right), \quad B_{m:n,\lambda}(\bx)=B_{m:n}(\bx)+\lambda \mathbf{I},
    \\
    \bC_{m:n}(\bx)&=\frac{2}{nh^{d}}\sum_{k=m+1}^{n}\bR_{k}(\bx)f(\bx_k), \quad \bG_{m:n}(\bx) = \frac{2}{nh^{d}}\sum_{k=m+1}^{n}\bR_{k}(\bx)\xi_k.
\end{align*}
The proofs will use the fact that Lemmas \ref{boundB} and \ref{boundC} apply not only to the vector $\bC_k(\bx)$ and matrix $B_{k,\lambda}(\bx)$ but also quite analogously to $\bC_{m:n}(\bx)$ and  $B_{m:n,\lambda}(\bx)$, so that under the assumptions of Theorem \ref{funcest}  we have
\begin{align}\label{eq:lem18_2n}
\sup_{\bx\in\com}\Exp\left[\norm{{B_{m:n,\lambda}(\bx)}^{-1}}_{\op}^{4}\right] \leq \cst \lambda_{\min}^{-4},
\end{align}
\begin{align}\label{eq:lem18_3n}
\sup_{\bx\in\com}\Exp\left[\norm{B_{m:n,\lambda}(\bx)^{-1} - (\Exp[B_{m:n,\lambda}(\bx)])^{-1}}_{\op}^2\right] \leq \cst h^{-d}n^{-1}.
\end{align}
\begin{align}\label{eq:lem:boundCXn}
   \sup_{\bx\in\com} \Exp \left[\norm{\bC_{m:n}(\bx) - \Exp\left[\bC_{m:n}(\bx)\right]}^4\right] \leq \cst h^{-2d}n^{-2}
\end{align}
and 
\begin{align}\label{eq:lem:boundCX-1n}
\sup_{\bx\in\com} \Exp \left[\norm{\bC_{m:n}(\bx))}^4\right] \leq \cst.
\end{align}

The following lemma establishes a bound on the bias of $f_n$.
\begin{lemma}\label{bias2}Under Assumptions \ref{ass:kernel} -- \ref{distass} we have
\begin{align*}
\sup_{\bx\in\com}
    \left|\Exp\left[f_n(\bx)\right] - f(\bx)\right| \leq \cst n^{-\frac{\beta}{2\beta+d}}\enspace.
\end{align*}
\end{lemma}
\begin{proof}
Set $E_{m:n}(\bx) = \Exp[B_{m:n}(\bx)]$. It follows from Lemma \ref{elbound}(ii) that matrix $E_{m:n}(\bx)$ is positive definite  for all $\bx\in \com$ and 
\begin{align}\label{eq:lambda_min}
\sup_{\bx\in\com}\norm{E_{m:n}(\bx)^{-1} }_{\op} \leq \lambda_{\min}^{-1},
\quad
  \sup_{\bx\in\com}\norm{(\Exp\left[{B_{m:n,\lambda}(\bx)}\right])^{-1} }_{\op} \leq \lambda_{\min}^{-1}. 
\end{align}
 Fix $\bx\in \com$ and introduce the notation $\bar\bphi_n = f_n(\bx) - \bU^{\top}(0) E_{m:n}(\bx)^{-1}B_{m:n}(\bx)\bc_h(f,\bx)$, where $\bc_h(f,\bx)$ is defined in the proof of Lemma \ref{bias1}. Using the fact that 
\begin{align} \label{eq:dopp1}
\Exp[\bU^{\top}(0)E_{m:n}(\bx)^{-1}B_{m:n}(\bx)\bc_h(f, \bx)]=\bU^{\top}(0)\bc_h(f,\bx)=f(\bx)
\end{align}
we have
$$
    \Exp\left[\bar\bphi_n\right] = \Exp\left[f_n(\bx)\right] - f(\bx).
$$
Thus, we need to control $|\Exp\left[\bar\bphi_n\right]|$. Introduce the notation
\begin{align*}
    \bar\bpsi_{1,n} &= \bU^{\top}(0) E_{m:n}(\bx)^{-1}\bC_{m:n}(\bx),
    \\
    \bar \bpsi_{2,n}&=\bU^{\top}(0)\left(E_{m:n}(\bx)+\lambda\mathbf{I}\right)^{-1}\bC_{m:n}(\bx).
\end{align*}
We have
\begin{align*}
    |\Exp[\bar\bphi_n]| \leq \underbrace{|\Exp[\bar\bpsi_{1,n} - \bU^{\top}(0)E_{m:n}(\bx)^{-1}B_{m:n}(\bx)\bc_h(f, \bx)]|}_{\text{term I}}+
    \underbrace{|\Exp[\bar\bpsi_{2,n} - \bar\bpsi_{1,n}]|}_{\text{term II}} + \underbrace{|\Exp[f_n (\bx)-  \bar\bpsi_{2,n}]|}_{\text{term III}}.
\end{align*}
Using the fact that $\norm{\bU(0)}=1$ the analysis of the three terms in this expression follows the same lines as the analysis of analogous terms 
in the proof of Lemma \ref{bias1}. The only essential difference is that factor $h^{-1}$ is now dropped. Thus, we get 
 \begin{align*}
    \text{term I} &\leq \cst h^{\beta}.
\end{align*}
The terms II and III are also evaluated in the same way as in Lemma \ref{bias1} (but with no $h^{-1}$ factor) by applying \eqref{eq:lem18_2n} --\eqref{eq:lem:boundCX-1n} instead of the analogous bounds from Lemmas \ref{boundB} and \ref{boundCX}. This yields
\begin{align*}
    \text{term II} \leq \cst \lambda\quad\text{and}\quad\text{term III}\leq \cst h^{-\frac{d}{2}}n^{-\frac{1}{2}}.
\end{align*}
Therefore,  
\begin{align*}
    |\Exp[\bar\bphi_n]| \leq \cst \left(h^{\beta} + \lambda+ h^{-\frac{d}{2}}n^{-\frac{1}{2}}\right)\enspace.
\end{align*}
Since $h = n^{-\frac{1}{2\beta+d}}$ and $\lambda = n^{-\frac{\beta}{2\beta+d}}$ the lemma follows.
\end{proof}
The next lemma establishes a bound on the variance of the regularized local polynomial estimator $f_n$.
\begin{lemma}\label{var}
Under Assumptions \ref{ass:kernel} -- \ref{distass} we have 
\begin{align*}
\sup_{\bx \in \com}
    \Exp\left[\left(f_n(\bx) - \Exp\left[f_n(\bx)\right]\right)^{2}\right] \leq \cst n^{-\frac{2\beta}{2\beta+d}}\enspace.
\end{align*}
\end{lemma}
\begin{proof}
Since $\bD_{m:n}(\bx)=\bG_{m:n}(\bx)+\bC_{m:n}(\bx)$ then, taking into account the facts that $\norm{\bU(0)}=1$ and 
\begin{align}
  \Exp\left[f_n(\bx)\right]= \Exp\left[\bU^\top (0)B_{m:n,\lambda}(\bx)^{-1}\bC_{m:n}(\bx)\right]  
\end{align}
we obtain
\begin{align*}
    \Exp[\left(f_n(\bx)-\Exp\left[f_n(\bx)\right]\right)^2] 
    &\leq  2\underbrace{\Exp\left[\norm{{B_{m:n,\lambda}(\bx)^{-1}}\bG_{m:n}(\bx)}^{2}\right]}_{\text{term I}} 
    \\
    & \quad + 2\underbrace{\Exp\left[\norm{{B_{m:n,\lambda}(\bx)^{-1}}\bC_{m:n}(\bx)- \Exp\left[{B_{m:n,\lambda}(\bx)^{-1}}\bC_{m:n}(\bx)\right]}^{2}\right]}_{\text{term II}}\enspace.
\end{align*}
Using Assumption \ref{assnoise} we get
\begin{align*}
    \text{term I} &\leq \cst \,\Exp\left[\norm{B_{m:n,\lambda}(\bx)^{-1} }_{\op}^{2}\left(n^{-2}h^{-2d}\sum_{k = m+1}^{n}\norm{\bR_k(\bx)}^2\right)\right].
\end{align*}
Applying the Cauchy-Schwarz inequality and \eqref{eq:lem18_2n} we find that 
\begin{align*}
    \text{term I} &\leq \cst \left(\Exp\left[\norm{B_{m:n,\lambda}(\bx)]}_{\op}^{-4}\right]\Exp\left[\left(n^{-2}h^{-2d}\sum_{k = m+1}^{n}\norm{\bR_k(\bx)}^2\right)^2\right]\right)^{\frac{1}{2}}
    \\
    & \leq \cst\lambda_{\min}^{-2}\left(\Exp\left[\left(n^{-2}h^{-2d}\sum_{k = m+1}^{n}\norm{\bR_k(\bx)}^2\right)^2\right]\right)^{\frac{1}{2}}
    \\
    &\leq \cst n^{-2}h^{-2d}\bigg(\sum_{k=m+1}^{n}\Exp[\norm{\bR_{k}(\bx)}^4] + \sum_{j,k=m+1, j\ne k}^{n}\Exp[\norm{\bR_{j}(\bx)}^2]\Exp[\norm{\bR_{k}(\bx)}^2]\bigg)^{\frac{1}{2}} 
    \\
    &\leq \cst n^{-2}h^{-2d}( n h^{d} + n^{2}h^{2d})^{{1}/{2}} \le \cst n^{-\frac{2\beta}{2\beta+d}},
\end{align*}
where in the last line we have used Lemma \ref{elbound}(i) and the fact that $h=n^{-\frac{1}{2\beta+d}}$.

Next, arguing analogously to \eqref{eq:dlinnoe} we get
\begin{align*}
\text{term II}
&\leq 3\,\Exp\left[\norm{{B_{m:n,\lambda}(\bx)}^{-1}\left(\bC_{m:n}(\bx) - \Exp\left[\bC_{m:n}(\bx)\right]\right)}^2\right]
\\
&\phantom{\leq}+ 6\, \Exp\left[\norm{{B_{m:n,\lambda}(\bx)}^{-1}-\left(\Exp\left[{B_{m:n,\lambda}(\bx)}\right]\right)^{-1}}_{\op}^{2}\right]\Exp\left[\norm{\bC_{m:n}(\bx)}^{2}\right]
\\
&\leq 3\,\left(\Exp\left[\norm{{B_{m:n,\lambda}(\bx)}^{-1}}^{4}_{\op}\right]\Exp\left[\norm{\bC_{m:n}(\bx) - \Exp\left[\bC_{m:n}(\bx)\right]}^{4}\right]\right)^{\frac{1}{2}} 
\\
&\phantom{\leq}+ 6\, \Exp\left[\norm{{B_{m:n,\lambda}(\bx)}^{-1}-\left(\Exp\left[{B_{m:n,\lambda}(\bx)}\right]\right)^{-1}}_{\op}^{2}\right]
\left(\Exp\left[\norm{\bC_{m:n}(\bx)}^{4}\right]\right)^{\frac{1}{2}} .
\end{align*}
Applying here the bounds \eqref{eq:lem18_2n} --\eqref{eq:lem:boundCX-1n} and using the fact that $h = n^{-\frac{1}{2\beta+d}}$ we obtain 
\begin{align*}
    \text{term II} \leq \cst n^{-\frac{2\beta}{2\beta+d}}.
\end{align*}
The proof is completed by putting together the bounds on terms I and II.
\end{proof}

\subsection{Proof of Theorem \ref{lower_bound}}\label{proof_lower_bound}

 We first prove \eqref{lower_minimum_value}.
We apply the scheme of proving lower bounds for estimation of functionals described in Section 2.7.4 in \cite{Tsybakov09}. Moreover, we use its basic form when the problem is reduced to testing two simple hypotheses (that is, the mixture measure $\mu$ from Section 2.7.4 in \cite{Tsybakov09} is the Dirac measure). Without loss of generality, assume that $\com$ is a sufficiently large Euclidean ball centered at 0. The functional we are estimating is $F(f)=f^*=\min_{\bx\in\com}f(\bx)$. We choose the two hypotheses as the probability measures $P_1^{\otimes n}$ and $ P_2^{\otimes n}$, where $P_j$ stands for the distribution of a pair $(\bx_i,y_i)$ satisfying \eqref{mainmodel} with $f=f_j$, $j=1,2$. For $r>0$, $\delta>0$, we set
$$
f_1(\bx) =\alpha(1+\delta)\|\bx\|^{2} / 2, \quad f_2(\bx)=f_1(\bx)+r h_{n}^{\beta} \Phi\left(\frac{\bx-\bx^{(n)}}{h_{n}}\right),
$$
where $h_n = n^{-1/(2\beta+d)},$ $\bx^{(n)} = (h_n/8,0,\dots,0) \in \mathbb{R}^d$ and $\Phi(\bx) = \prod_{i=1}^d \Psi(x_i)$ with $$\Psi(t) = \int_{-\infty}^t \left(\eta(y+1/2)-\eta(y)\right)\drm y,$$
where $\eta(\cdot)$ is an infinitely many times differentiable function on $\mathbb{R}^{1}$ such that
$$
\eta(x) \geq 0, \quad \eta(x)= \begin{cases}0, & x \notin[0,1 / 2] \\ 1, & x \in[1 / 8,3 / 8]\end{cases}.
$$
{First, assume that $\alpha \geq n^{-(\beta-2)/(2\beta+d)}$.} It is shown in \cite{tsybakov1990passive} that
 if $r$ is small enough the functions $f_1$ and $f_2$ are $\alpha$-strongly convex and belong to $\mathcal{F}_{\beta}(L).$ Thus, $
      f_{j} \in \mathcal{F}_{\beta, \alpha}(L), j=1,2.
        $
        It is  also not hard to check that for the function $\eta_1(y)=\eta(y+1/2)-\eta(y)$ we have
        $$
        \eta_1\left(-\frac{r \Psi^{d-1}(0)h_n^{\beta-2}}{\alpha(1+\delta)}-\frac{1}{8}\right)=1,
        $$
{due to the fact that $r\alpha^{-1}(1+\delta)^{-1}h_n^{\beta-2}<1/4$, for $\alpha \geq n^{-(\beta-2)/(2\beta+d)}$ and sufficiently small values of $r$ and $\delta$.} Using this remark we get that
the minimizers $\bx^*_j = \argmin_{\bx\in \com} f_j(\bx)$ have the form
\[
\bx_1^*=(0,0,\dots,0) \quad \text{and} \quad  \bx^*_2 = \left(-\frac{r \Psi^{d-1}(0) h_{n}^{\beta-1}}{\alpha(1+\delta)}, 0, \ldots, 0\right).
\] 
The values of the functional $F$ on $f_1$ and $f_2$ are $F(f_1)=0$ and 
\begin{align*}
    F(f_2) & = f_2(\bx^*_2)\\
    & = \frac{r^2 \Psi^{2(d-1)}(0)}{2\alpha (1+\delta)}h_n^{2(\beta - 1)} + r \Psi^{d-1}(0) \Psi\left(-\frac{r \Psi^{d-1}(0)h_n^{\beta-2}}{\alpha(1+\delta)}-\frac{1}{8}\right)h_n^{\beta} \\
    & \ge \frac{r^2 \Psi^{2(d-1)}(0)}{2\alpha (1+\delta)}h_n^{2(\beta - 1)} + r \Psi^{d-1}(0) \Psi(-1/4) h_n^{\beta} \quad (\text{for } r \text{ small \ enough}) \\
    & \geq r \Psi^{d-1}(0) \Psi(-1/4) h_n^{\beta}.
\end{align*}
Here, $\Psi(0)=\int_{-\infty}^{\infty} \eta(y)\drm y >0$ and 
$\Psi(-1/4)=\int_{-\infty}^{1/4} \eta(y)\drm y >0$. 

Note that assumption (i) of Theorem 2.14 in \cite{Tsybakov09} is satisfied with $\beta_0=\beta_1=0$, $c=0$ and $s=r \Psi^{d-1}(0) \Psi(-1/4) h_n^{\beta}/2$. Therefore, by Theorem 2.15 (ii) in \cite{Tsybakov09}, \eqref{lower_minimum_value} will be proved if we show that 
\begin{equation}\label{helling}
\operatorname{H}^2\left(P_1^{\otimes n}, P_2^{\otimes n}\right) \le a<2,
\end{equation}
where $\operatorname{H}^2\left(P, Q\right)$ denotes the Hellinger distance between the probability measures $P$ and $Q$. 
Using assumption (\ref{noice_assumption}) we obtain
\begin{align*}
\operatorname{H}^2\left(P_1^{\otimes n}, P_2^{\otimes n}\right) &=2\left(1-\left(1-\frac{\operatorname{H}^2(P_1,P_2)}{2} \right)^n \right)\\
 \quad &\leq n\operatorname{H}^2(P_1,P_2)\quad  (\text{as} \,\, (1-x)^{n} \geq 1-x n, \ x\in [0,1])\\
& = n \int \left(\sqrt{p_{\xi}(y)} - \sqrt{p_{\xi}\left(y+\left(f_1(\bx)-f_2(\bx)\right)\right)} \right)^2 p(\bx)\drm\bx \drm y \\
 &\leq n I_{*} \int\left(f_1(\bx)-f_2(\bx)\right)^{2} p(\bx) \drm \bx \\
 & = n I_{*} r^{2} h_{n}^{2 \beta+d} \int \Phi^{2}(\bu) p\left(\bx^{(n)}+\bu h_{n}\right) \drm\bu  \\
& \leq p_{\max } I_{*} r^{2} \int \Phi^{2}(\bu) \drm\bu,  \quad \text{for} \,\,\, r \leq v_0,
\end{align*}
where $p_{\max }$ is the maximal value of the density $p(\cdot)$  of $\bx_i$.
Choosing $r \leq \sqrt{a /\left(p_{\max } I_{*} \int \Phi^{2}(\bu) \drm\bu\right)},$ with $a <2$ we obtain \eqref{helling}. This completes the proof of the lower bound \eqref{lower_minimum_value} for $\alpha\geq n^{-(\beta-2)/(2\beta+d)}$. If $0 < \alpha < \alpha_0 := n^{-(\beta-2)/(2\beta+d)}$ the same lower bound holds due to the nesting property of the classes $\mathcal{F}_{\alpha,\beta}(L)$. Indeed,  $\mathcal{F}_{\alpha_0,\beta}(L) \subset \mathcal{F}_{\alpha,\beta}(L)$ for $0 < \alpha < \alpha_0$. The proof of \eqref{lower_minimum_value} is now complete. 

In order to prove \eqref{lower_minimizer}, it suffices to use the same construction of two hypotheses as above, apply the Hellinger version of Theorem 2.2 from \cite{Tsybakov09} and notice that, for $\alpha\geq n^{-(\beta-2)/(2\beta+d)}$,
\begin{align}\label{eq:lower_B}
   \Vert \bx_1^* - \bx^*_2\Vert \geq \cst \alpha^{-1}h_n^{\beta-1}.
\end{align}
This proves the lower bound \eqref{lower_minimizer} for $\alpha\geq \alpha_0$, with the normalizing factor $\alpha n^{(\beta-1)/(2\beta+d)}$. Notably, if $\alpha = \alpha_0$ this factor is equal to $n^{1/(2\beta+d)}$. Using the inclusion $\mathcal{F}_{\alpha_0,\beta}(L) \subset \mathcal{F}_{\alpha,\beta}(L)$ valid for $0<\alpha <\alpha_0$ we conclude that the bound \eqref{lower_minimizer} with the normalizing factor $n^{1/(2\beta+d)}$ holds true for all such values of $\alpha$. This completes the proof of \eqref{lower_minimizer}.

Finally, the lower bound \eqref{lower_opt_error} follows immediately from \eqref{lower_minimizer} with $w(u)=u^2$ and the inequality $f(\hat\bx_n)-f^*\ge (\alpha/2)\norm{\hat\bx_n-\bx^*}^2$ granted by the $\alpha$-strong convexity of $f$.

\subsection{Proof of Theorem \ref{th:lower_bound_nonconvex}}
\label{sec:proof_lower_bound_nonconvex}

We apply again the scheme of proving lower bounds for estimation of functionals from Section 2.7.4 in \cite{Tsybakov09}. However, we use a different construction of the hypotheses. Without loss of generality, assume that $n\ge 2$ and that $\com$ contains the cube $[0,1]^d$. Define $h_n = (n/\log(n))^{-1/(2\beta+d)}$, $N=(1/h_n)^d$, and assume without loss of generality that $N$ is an integer. 
 For $r>0$, we set
$$
 f_j(\bx)=- r h_{n}^{\beta} \Phi\left(\frac{\bx-{\bf t}^{(j)}}{h_{n}}\right), \quad j=1,\dots,N,
$$
where  $\Phi(\bx) = \prod_{i=1}^d \Psi(x_i)$,
where $\Psi(\cdot)$ is an infinitely many times differentiable function on $\mathbb{R}$ taking positive values on its support $[-1/2,1/2]$,
and we denote by ${\bf t}^{(1)},\dots,{\bf t}^{(N)} $ the $N$ points of the equispaced grid on $[0,1]^d$ with step $h_n$ over each coordinate, such that the supports of all $f_j$'s are included in $[0,1]^d$ and are disjoint. It is not hard to check that for $r$ small enough all the functions $f_j$, $ j=1,\dots,N$, belong to $\mathcal{F}_{\beta}(L)$.

We consider the product probability measures $P_0^{\otimes n}$ and $ P_1^{\otimes n}, \dots P_N^{\otimes n}$, where $P_0$ stands for the distribution of a pair $(\bx_i,y_i)$ satisfying \eqref{mainmodel} with $f\equiv 0$, and $P_j$ stands for the distribution of $(\bx_i,y_i)$ satisfying \eqref{mainmodel} with $f=f_j$. Consider the mixture probability measure ${\mathbb P}_\mu=\frac1N \sum_{j=1}^{N} P_j^{\otimes n}$, where $\mu$ denotes the uniform distribution on $\{1,\dots,N\}$.

Note that, for each $j=1,\dots,N$, we have  $F(f_j)=-r h_{n}^{\beta}\Phi_{\max}$, where $F(f)=f^*=\min_{\bx\in\com}f(\bx)$, and  $\Phi_{\max}>0$ denotes the maximal value of function $\Phi(\cdot)$. Let $$\chi^2(P',P)=\int (\drm P'/\drm P)^2 \drm P -1$$ 
denote the chi-square divergence between two mutually absolutely continuous probability measures $P'$ and~$P$. We will use the following lemma, which is a special case of Theorem~2.15 in \cite{Tsybakov09}. 
\begin{lemma}\label{lem:lower_tsybakov_book}
	Assume that there exist $v>0, b>0$ such that $F(f_j) = - 2v$ for $j=1,\dots,N$ and $\chi^2({\mathbb P}_{\mu} ,P_0^{\otimes n}) \le b$, 
	Then
	\begin{equation*}
	\inf_{\hat{f}_n}\sup_{j=0,1,\dots,N} P_j^{\otimes n}\big(|\hat{f}_n-F(f_j)|\ge v\big) \ge \frac14\exp(-b),
	\end{equation*}
	where $\inf_{\hat{f}_n}$ denotes the infimum over all estimators.
\end{lemma}
 In our case, the first condition of this lemma is satisfied with $v=r h_{n}^{\beta}\Phi_{\max}/2$. We now check that the second condition 
$\chi^2({\mathbb P}_{\mu} ,P_0^{\otimes n}) \le b$ holds with some constant $b>0$ independent of $n$. Using a standard representation of the chi-square divergence of a Gaussian mixture  from the pure Gaussian noise measure (see, for example, Lemma 8 in \cite{Carpentier-etal2019}) we obtain
\begin{align*}
	\chi^2({\mathbb P}_{\mu} ,P_0^{\otimes n}) &= \frac1{N^2} \sum_{j,j'=1}^N 
	{\mathbf E} \exp\left(\frac{\sum_{i=1}^{n} f_j(\bx_i)f_{j'}(\bx_i)}{\sigma^2}\right) -1
	\\
	&= \frac1{N^2} \sum_{j,j'=1}^N 
	{\mathbf E} \exp\left(\frac{\sum_{i=1}^{n} f_j(\bx_i)f_{j'}(\bx_i)}{\sigma^2}\right) -1 
	\\
	&= \frac1{N^2} \sum_{j=1}^N 
	{\mathbf E} \exp\left(\frac{\sum_{i=1}^{n} f_j^2(\bx_i) }{\sigma^2}\right) + \frac{N(N-1)}{N^2}-1 
	\\
	&\le \frac1{N^2} \sum_{j=1}^N 
	{\mathbf E} \exp\left(\frac{\sum_{i=1}^{n} f_j^2(\bx_i) }{\sigma^2}\right) 
	\\
	&= \frac1{N^2} \sum_{j=1}^N 
\left[	{\mathbf E} \exp\left(\frac{f_j^2(\bx_1) }{\sigma^2}\right) \right]^n,
\end{align*}
where the equality in the third line is due to the fact that if $j\ne j'$ then $f_j$ and $f_{j'}$ have disjoint supports and thus $f_j(\bx_i)f_{j'}(\bx_i)=0$. Note that $\max_{\bx\in \mathbb{R}^d} f_j^2(\bx)\le r^2 \Phi_{\max}^2,$ for all $j=1,\dots,N$.
Choose $r$ such that $r\le \sigma/\Phi_{\max}$. Then $\frac{f_j^2(\bx_1) }{\sigma^2}\le 1$, and using the elementary inequality $\exp(u)\le 1+2u, u\in [0,1],$ we obtain that
$\exp\left(\frac{f_j^2(\bx_1) }{\sigma^2}\right)\le 1 + \frac{2 f_j^2(\bx_1) }{\sigma^2}$ for all $j=1,\dots,N$. Substituting this bound in the last display and noticing that $	{\mathbf E} (f_j^2(\bx_1))=\int f_j^2(\bx) p(\bx) \drm \bx \le p_{\max}r^2h_n^{2\beta+d}\int \Phi^2(\bx) \drm \bx= c_*\frac{\log n}{n}$, where $c_*= p_{\max}r^2\int \Phi^2(\bx)\drm \bx$, we obtain:
$$
\chi^2({\mathbb P}_{\mu} ,P_0^{\otimes n})\le \frac1{N} 
\left[1 +\frac{2 	{\mathbf E} (f_j^2(\bx_1) )}{\sigma^2} \right]^n \le \frac1{N} 
\left[1 +\frac{2 c_*	\log n}{\sigma^2 n} \right]^n \le \frac1{N} \exp\left(\frac{2 c_*	\log n}{\sigma^2}\right)= \frac{n^{c_0}}{N},
$$
where $c_0=2c_*/\sigma^2=2p_{\max}r^2\int \Phi^2(\bx)\drm \bx/\sigma^2$. Since $N=(n/\log n)^{\frac{d}{2\beta+d}}$ we finally get
	$$
	\chi^2({\mathbb P}_{\mu} ,P_0^{\otimes n})\le n^{c_0-\frac{d}{2\beta+d}} (\log n)^{\frac{d}{2\beta+d}}.
	$$
By choosing $r$ small enough to have $c_0\le \frac{d}{2(2\beta+d)}$ we obtain that $\chi^2({\mathbb P}_{\mu} ,P_0^{\otimes n})\le \left(\frac{\log n}{\sqrt{n}}\right)^{\frac{d}{2\beta+d}}\le \left(\frac{\log 2}{\sqrt{2}}\right)^{\frac{d}{2\beta+d}}: = b$. Thus, the second condition of Lemma \ref{lem:lower_tsybakov_book} holds if $r$ is chosen as a small enough constant. Notice that, in Lemma \ref{lem:lower_tsybakov_book}, the rate $v$ is of the desired order
$(n/\log n)^{-\frac{\beta}{2\beta+d}}$. The result of the theorem now follows from Lemma \ref{lem:lower_tsybakov_book} and the standard argument to obtain the lower bounds, see Section 2.7.4 in \cite{Tsybakov09}. 

\vskip 0.2in
\bibliography{sample}

\appendix

\section{Auxiliary lemmas}\label{app}
Recall that $\com' = \{\bx + \by: \bx\in \com\quad \text{and}\quad \Vert\by\Vert \le 1 \} \supseteq \{\bx + \by: \bx\in \com\quad \text{and}\quad \by\in \Supp(K)\}$.
\begin{lemma}\label{elbound}
For any $q \geq 1$ let 
\begin{align*}
    \nu_{q} = \int_{\mathbb{R}^d}\norm{\bU(\bu)K(\bu)}^{q}\drm\bu 
\end{align*}
and $p_{\max} = \max_{\by \in \com'}p(\by)$. Under Assumptions \ref{ass:kernel} and \ref{distass} for any $\bx\in \com$, $k\in[n]$ and $i\in [k]$ we have
\begin{itemize}
        \item[(i)] $\sup_{\bx\in\com}\Exp\left[\norm{\bR_{i,k}(\bx)}^{q}\right] \leq p_{\max}\nu_{q} h_k^{d}  \ $,  and  $ \ \sup_{\bx\in\com}\Exp\left[\norm{\bR_{k}(\bx)}^{q}\right] \leq p_{\max}\nu_{q} h_{m:n}^{d}$.
        \item[(ii)] There exists a constant $\lambda_{\min}>0$ such that 
        $$\displaystyle{\inf_{\bx\in\com}}\lambda_{\min}\left(E_{k}(\bx)\right)\geq \lambda_{\min} \quad \text{and} \quad \displaystyle{\inf_{\bx\in\com}}\lambda_{\min}\left(E_{m:n}(\bx)\right) \geq \lambda_{\min}.$$
\end{itemize}
\end{lemma}
\begin{proof}
We have
 \begin{align*}
    h_k^{-d}\sup_{\bx\in\com}\Exp\left[\norm{\bR_{i,k}(\bx)}^{q}\right] 
    &= h_k^{-d}\sup_{\bx\in\com}\int_{\mathbb{R}^d}\norm{\bU\left(\frac{\by - \bx}{h_k}\right)K\left(\frac{\by - \bx}{h_k}\right)}^{q}p(\by)\drm \by
    \\&\leq \int_{\mathbb{R}^d}\norm{\bU(\bu)K(\bu)}^{q}\sup_{\bx\in\com}p(\bx + h_k\bu)\drm\bu\leq p_{\max}\nu_{q}.
\end{align*}
The bound for $\sup_{\bx\in\com}\Exp\left[\norm{\bR_{k}(\bx)}^{q}\right]$ is 
proved exactly in the same way. 
To prove (ii) note that
\begin{align*}
    E_{k}(\bx) = \Exp\left[B_{k}(\bx)\right] 
    &= h_k^{-d}\Exp\left[\bU\left(\frac{\bx_1 - \bx}{h_k}\right)\bU^{\top}\left(\frac{\bx_1 - \bx}{h_k}\right)K\left(\frac{\bx_1 - \bx}{h_k}\right)\right]
    \\&= \int_{\mathbb{R}^d}\bU\left(\bu\right)\bU^{\top}\left(\bu\right)K\left(\bu\right)p(\bx+h_k\bu)\drm \bu
\end{align*}
and thus $\inf_{\bx\in\com}\lambda_{\min}\left(E_{k}(\bx)\right) \geq p_{\min}\lambda_{\min}\left(H\right)$, where $H = \int_{\mathbb{R}^d}\bU\left(\bu\right)\bU^{\top}\left(\bu\right)K\left(\bu\right)\drm \bu$. Arguing as in \cite[Lemma 1]{Tsy86} we obtain that $\lambda_{\min}\left(H\right)>0$. This gives the desired bound for $\inf_{\bx\in\com}\lambda_{\min}\left(E_{k}(\bx)\right)$ with $\lambda_{\min} = p_{\min}\lambda_{\min}\left(H\right)$. The bound $\displaystyle{\inf_{\bx\in\com}}\lambda_{\min}\left(E_{m:n}(\bx)\right)\geq \lambda_{\min}$ is proved exactly in the same way. \end{proof}
\begin{lemma}\label{boundB}
Let Assumptions \ref{ass:kernel} and  \ref{distass} hold. Then there exists a constant $\cst>0$ such that for $kh_k^d\ge 1$ we have
\begin{align}\label{eq:lem18_1}
    \sup_{\bx\in\com}\Exp\left[\norm{B_{k,\lambda}(\bx) - \Exp\left[B_{k,\lambda}(\bx)\right]}^{4}_{\op}\right] \leq \cst h_k^{-2d}k^{-2}.  
\end{align}
Furthermore, for  $kh_k^d \geq \lambda^{-2}$ we have
\begin{align}\label{eq:lem18_2}
\sup_{\bx\in\com}\Exp\left[\norm{{B_{k,\lambda}(\bx)}^{-1}}_{\op}^{4}\right] \leq \cst \lambda_{\min}^{-4}
\end{align}
and
\begin{align}\label{eq:lem18_3}
\sup_{\bx\in\com}\Exp\left[\norm{B_{k,\lambda}(\bx)^{-1} - (\Exp[B_{k,\lambda}(\bx)])^{-1}}_{\op}^2\right] \leq \cst h_k^{-d}k^{-1}.
\end{align}
\end{lemma}
\begin{proof}
For $s,s'\in [S]$, consider the function $G^{(ss')}:\mathbb{R}^d\to \mathbb{R}$ such that 
\begin{align*}
   G^{(ss')}(\bu) = \left(\bU\left(\bu\right)\bU^{\top}\left(\bu\right)\right)_{ss'}, \quad \bu\in \mathbb{R}^d,
\end{align*} 
where, for a matrix $A$, we denote by $(A)_{ss'}$ its $(s,s')$-entry. 
Set 
$$F_i^{(s,s')}(\bx)=G^{(ss')}(\frac{\bx_i - \bx}{h_k})K(\frac{\bx_i - \bx}{h_k}) - \Exp\left[G^{(ss')}(\frac{\bx_i - \bx}{h_k})K(\frac{\bx_i - \bx}{h_k})\right].
$$ 
The inequality between the operator norm and the $\ell_1$-norm of a matrix yields that
    \begin{align}\label{eq:B1}
    \norm{B_{k,\lambda}(\bx) - \Exp\left[B_{k,\lambda}(\bx)\right]}_{\op} 
    \leq \sum_{s,s'=1}^{S}\left|k^{-1}h_k^{-d}\sum_{i=1}^{k}F_i^{(s,s')}(\bx)\right|,
\end{align}
so that 
\begin{align*}
    \Exp\left[\norm{B_{k,\lambda}(\bx) - \Exp\left[B_{k,\lambda}(\bx)\right]}^{4}_{\op}\right] 
    &\leq S^6
    \sum_{s,s'=1}^{S}\Exp\left[\left|k^{-1}h_k^{-d}\sum_{i=1}^{k}F_i^{(s,s')}(\bx)\right|^4\right]
    \\
    & = 4 S^6
    \sum_{s,s'=1}^{S}\int_0^\infty t^3 \Prob  \left[\left|k^{-1}h_k^{-d}\sum_{i=1}^{k}F_i^{(s,s')}(\bx)\right|>t\right] \drm t.
\end{align*}
To control the probabilities on the right hand side, we use Bernstein's inequality that we recall here for convenience.

\begin{lemma}\label{lem:Bern} Let $\{\zeta_i\}_{i=1}^{k}$ be a collection of real valued  independent random variables with zero means such that, for any $i=1,\dots,k$ and any $q\in  \{2,3,\dots \},$  the Bernstein condition $\Exp[|\zeta_i|^q]\le (q!/2) \nu H^{q-2}$ is satisfied with some 
$\nu>0, H>0$. Then, for any $t>0$,
\begin{align*}
&\Prob\left(\left|\frac1k\sum_{i=1}^{k}\zeta_i\right|> t\right) \le 2\exp\left(-\frac{t^2k}{2(\nu + Ht)}\right).
\end{align*}
\end{lemma}

We apply Lemma \ref{lem:Bern} with $\zeta_i=h_k^{-d}F_i^{(s,s')}(\bx)$.  Let $q\ge 2$. We have $F_i^{(s,s')}(\bx)=X-\Exp[X]$ with $X=G^{(ss')}(\frac{\bx_i - \bx}{h_k})K(\frac{\bx_i - \bx}{h_k})$. Using the inequalities  
$\Exp[|X-\Exp[X]|^q] \le 2^{q-1}(\Exp[|X|^q]+ |\Exp[X]|^q)\le 2^q \Exp[|X|^q]$ we obtain that,
for any $\bx\in \com$,
\begin{align*}
\Exp[|F_i^{(s,s')}(\bx)|^q]&\le 2^q
\Exp\left[\left|G^{(ss')}(\frac{\bx_i - \bx}{h_k})K(\frac{\bx_i - \bx}{h_k})\right|^{q}\right] 
\\
    &= 2^q \int \left|G^{(ss')}(\frac{\bw - \bx}{h_k})K(\frac{\bw - \bx}{h_k})\right|^{q} p(\bw)\drm \bw 
    \\
    &\le   p_{\max} 2^q h_k^{d} \int \left|G^{(ss')}(\bu)K(\bu)\right|^{q} \drm \bu 
    \leq \cst_0 \cst_1^q h_k^{d},   
\end{align*}
where we have used the fact that the functions $G^{(ss')}$ and $K$ are uniformly bounded on the support of $K$. Here, the constants $\cst_0,\cst_1$ do not depend on $\bx,s,s'$.  Thus, for any $i=1,\dots,k$ and any $q\in  \{2,3,\dots \},$ 
$$
\Exp[|\zeta_i|^q]\le  \nu H^{q-2} \quad \text{with} \ \nu=\cst_0\cst_1^2 h^{-d}_k, \quad H = \cst_1 h^{-d}_k.
$$
Applying Lemma \ref{lem:Bern} with these values of $\nu$ and $H$ we get that there exists a constant $\cst_2>0$ such that,
 for any  $t>0$ and any $\bx\in \com$,
\begin{align}
 \label{eq:B2}
\Prob\left[\left|k^{-1}h_k^{-d}\sum_{i=1}^{k}F_i^{(s,s')}(\bx)\right|\geq t\right] 
&\leq 2\exp\left( -\frac{t^2kh_k^d}{\cst_2(1 + t)}\right).
\end{align}
Consequently,
\begin{align*}
    \Exp\left[\norm{B_{k,\lambda}(\bx) - \Exp\left[B_{k,\lambda}(\bx)\right]}^{4}_{\op}\right] 
    &\leq 
    8 S^8
    \int_0^\infty t^3 \exp\left( -\frac{t^2kh_k^d}{\cst_2(1 + t)}\right) \drm t
    \\
    &\leq 
    \cst\left(\int_0^1 t^3 \exp\left( -\frac{t^2kh_k^d}{2\cst_2}\right) \drm t + 
    \int_1^\infty t^3 \exp\left( -\frac{t kh_k^d}{2\cst_2}\right) \drm t\right)
    \\
    &\leq 
    \cst_3 h_k^{-2d} k^{-2},
\end{align*}
where we have used the assumption $kh_k^d\ge 1$ and $\cst_3$ is a constant that does not depend on $\bx$. 
This proves \eqref{eq:lem18_1}. 

To prove \eqref{eq:lem18_2}, note that Lemma \ref{elbound}(ii) implies the inequality $\sup_{\bx\in\com}\norm{\left(\Exp\left[B_{k}(\bx)\right]\right)^{-1}}_{\op}\leq \lambda_{\min}^{-1}$. Thus, also 
\begin{align}\label{eq:lem18_4}
\sup_{\bx\in\com}\norm{\left(\Exp\left[B_{k,\lambda}(\bx)\right]\right)^{-1}}_{\op}\leq \lambda_{\min}^{-1},     
\end{align}
while $\sup_{\bx\in\com}\norm{B_{k,\lambda}(\bx)^{-1}}_{\op}\leq \lambda^{-1}$. Therefore,
\begin{align}
    \Exp\left[\norm{B_{k,\lambda}(\bx)^{-1}}_{\op}^4\right] &\leq 8\Exp\left[\norm{B_{k,\lambda}(\bx)^{-1} - (\Exp[B_{k,\lambda}(\bx)])^{-1}}_{\op}^4\right] + \frac{8}{\lambda_{\min}^4}\nonumber
    \\&\leq 8\lambda_{\min}^{-4}\lambda^{-4}\Exp\left[\norm{B_{k,\lambda}(\bx)-\Exp\left[B_{k,\lambda}(\bx)\right]}_{\op}^{4}\right] + 8\lambda^{-4}_{\min}  \label{eq:lem18_5}
    \\&\leq \frac{8 \cst_3}{\lambda^{4}\lambda_{\min}^{4}}h_k^{-2d}k^{-2}+ \frac{8}{\lambda_{\min}^4},\nonumber
\end{align}
so that for $kh_k^d \geq \lambda^{-2}$ we obtain \eqref{eq:lem18_2}.

Finally, we prove \eqref{eq:lem18_3}. Using \eqref{eq:lem18_4} and then the Cauchy-Schwarz inequality we obtain  
\begin{align}\label{eq:lem18_6}
\begin{aligned}
 &\Exp\left[\norm{B_{k,\lambda}(\bx)^{-1} - \left(\Exp[B_{k,\lambda}(\bx)]\right)^{-1}}_{\op}^2\right]\\
&\le
    \Exp\bigg[\norm{{B_{k,\lambda} (\bx)^{-1}}}_{\op}^2\norm{(\Exp[B_{k,\lambda}(\bx)])^{-1}}_{\op}^2\norm{B_{k,\lambda}(\bx) - \Exp[B_{k,\lambda}(\bx)]}_{\op}^2\bigg]  
    \\
   &\le \lambda_{\min}^{-2} \left(\Exp\left[\norm{{B_{k,\lambda} (\bx)^{-1}}}_{\op}^{4}\right]\Exp\left[\norm{B_{k,\lambda}(\bx) - \Exp[B_{k,\lambda}(\bx)]}_{\op}^{4}\right]\right)^{\frac{1}{2}},   
\end{aligned}
\end{align}
so that \eqref{eq:lem18_3} follows by applying \eqref{eq:lem18_1} and \eqref{eq:lem18_2}.
\end{proof}
\begin{lemma}\label{boundBsup}
Let $k \in [n]$, and $h_k = \left(\frac{\log(k+1)}{k}\right)^{\frac{1}{2\beta+d}}$. Let Assumptions \ref{ass:kernel} and  \ref{distass} hold. Then there exists a constant $\cst>0$ such that
\begin{align}\label{first_sup}
    \Exp\left[\sup_{\bx\in\com}\norm{B_{k,\lambda}(\bx) - \Exp\left[B_{k,\lambda}(\bx)\right]}^{4}_{\op}\right] \leq \cst h_k^{-2d}k^{-2}\log(k+1)^{2}.
\end{align}
Furthermore, for $kh_k^d \geq \lambda^{-2}\log(k+1)$ we have
\begin{align}\label{second_sup}
\Exp\left[\sup_{\bx\in\com}\norm{{B_{k,\lambda}(\bx)}^{-1}}_{\op}^{4}\right] \leq \cst \lambda_{\min}^{-4}.
\end{align}
and 
\begin{align}\label{third_sup}
    \Exp\left[\sup_{\bx\in\com}\norm{B_{k,\lambda}(\bx)^{-1} - (\Exp\left[B_{k,\lambda}(\bx)\right])^{-1}}^{2}_{\op}\right] \leq \cst h_k^{-d}k^{-1}\log(k+1).
\end{align}
\end{lemma}
\begin{proof}
Notice that, due to Assumption \ref{ass:kernel} and the choice of $h_k$, the functions $\bx \mapsto G^{(ss')}(\frac{\bx_i - \bx}{h_k})K(\frac{\bx_i - \bx}{h_k})$ are supported on the compact $\Theta'$ and are Lipschitz continuous with Lipschitz constant $\cst_1 h_k^{-1}$, where the $\cst_1>0$ is independent of  $\bx_i$'s and also on $s,s'$ (by taking the maximum over $s,s'\in [S]$). It follows that there exists a constant $\cst$ such that for all $\bx, \by \in \com$ and all $s,s'\in [S]$ we have 
\begin{align}\label{eq:B3}
    \left|\sum_{i=1}^{k}(F_i^{(s,s')}(\bx)-F_i^{(s,s')}(\by)) \right|\leq \cst k h_k^{-1}\norm{\bx - \by}.
\end{align}
Set for brevity $\tilde\bB(\bx)=B_{k,\lambda}(\bx) - \Exp\left[B_{k,\lambda}(\bx)\right]$. Using \eqref{eq:B1} we obtain
\begin{align*}
    \Exp\left[\sup_{\bx\in\com}\norm{\tilde\bB(\bx)}^{4}_{\op}\right] 
    &\leq S^6
    \sum_{s,s'=1}^{S}\Exp\left[\sup_{\bx\in\com}\left|k^{-1}h_k^{-d}\sum_{i=1}^{k}F_i^{(s,s')}(\bx)\right|^4\right].
\end{align*}
For $\epsilon>0$, let $\mathcal{N}$ be the minimal $\epsilon$-net of $\com$ with respect to the Euclidean norm.  Using \eqref{eq:B3} we get
\begin{equation}\label{eq:B4}
\begin{aligned}
    \Exp\left[\sup_{\bx\in\com}\norm{\tilde\bB(\bx)}_{\op}^4\right] &\leq 8S^6
    \sum_{s,s'=1}^{S}\Exp\left[\sup_{\bx\in\mathcal{N}}\left|k^{-1}h_{k}^{-d}\sum_{i=1}^{k}F_i^{(s,s')}(\bx)\right|^4\right]  \\ 
    & +\frac{8S^6}{k^{4}h_{k}^{4d}}
    \sum_{s,s'=1}^{S}\Exp\left[\sup_{\bx,\by:\norm{\bx-\by}\leq \epsilon}\left|\sum_{i=1}^{k}\left(F_i^{(s,s')}(\bx)-F_i^{(s,s')}(\by)\right)\right|^4\right] 
    \\&\leq 8 S^6
    \sum_{s,s'=1}^{S}\Exp\left[\sup_{\bx\in\mathcal{N}}\left|k^{-1}h_{k}^{-d}\sum_{i=1}^{k}F_i^{(s,s')}(\bx)\right|^4\right] + \cst h_k^{-4d-4}\epsilon^4.
\end{aligned}
\end{equation}
We now provide an upper bound for the last sum in \eqref{eq:B4}. Denote by $\mathcal{N}\left(\com, \epsilon\right)$ the cardinality of $\mathcal{N}$. Using a standard bound $\mathcal{N}\left(\com, \epsilon\right)\le \left(\frac{\text{diam}(\com)}{\epsilon}+1\right)^{d}$ we find that, for any $t>0$,
\begin{align*}
\Prob\left[\sup_{\bx\in\mathcal{N}}\left|k^{-1}h_k^{-d}\sum_{i=1}^{k}F_i^{(s,s')}(\bx)\right|\geq t\right]&\leq \mathcal{N}(\com,\epsilon)\sup_{\bx\in\mathcal{N}}\Prob\left[\left|k^{-1}h_k^{-d}\sum_{i=1}^{k}F_i^{(s,s')}(\bx)\right|\geq t\right]
    \\&\leq
\left(\frac{\text{diam}(\com)}{\epsilon}+1\right)^{d}\sup_{\bx\in\mathcal{N}}\Prob\left[\left|k^{-1}h_k^{-d}\sum_{i=1}^{k}F_i^{(s,s')}(\bx)\right|\geq t\right].
\end{align*} 
Combining the last inequality with \eqref{eq:B2} we get that, for any  $t>0$,
\begin{align*}
\Prob\left[\sup_{\bx\in\mathcal{N}}\left|k^{-1}h_k^{-d}\sum_{i=1}^{k}F_i^{(s,s')}(\bx)\right|\geq t\right] &\leq 
2\exp\left(-\frac{t^2kh_k^d}{\cst(1 + t)} + d\log\left(\frac{\text{diam}(\com)}{\epsilon}+1\right)\right)
      \enspace.
\end{align*}
Set $\epsilon = \text{diam}(\com)h_k^{\frac{d}{2}+1}k^{-\frac{1}{2}}$. Then, for some constant $\cst_3>0$, 
\begin{align}\label{eq:B5}
\Prob\left[\sup_{\bx\in\mathcal{N}}\left|k^{-1}h_k^{-d}\sum_{i=1}^{k}F_i^{(s,s')}(\bx)\right|\geq t\right] &\leq 
      2\exp\left(-\frac{t^2kh_k^d}{\cst_3(1 + t)} + \cst_3 \log\left(k+1\right)\right).
\end{align}
Let $a = \cst_3( 2 k^{-1}h_k^{-d}\log(k+1))^{\frac{1}{2}}$, where $\cst_3$ is the constant from \eqref{eq:B5}. Noticing that $k h_k^{d}\ge \log(k+1)$ and using \eqref{eq:B5} we obtain 
\begin{equation}\label{eq:B6}
\begin{aligned}
\sum_{s,s'=1}^{S}\Exp\left[\sup_{\bx\in\mathcal{N}}\left|k^{-1}h_{k}^{-d}\sum_{i=1}^{k}F_i^{(s,s')}(\bx)\right|^4\right] 
&\leq S^2a^4 + 4S^2\int_{a}^{\cst_3\sqrt{2}}t^3\exp\left( - \frac{t^2kh_k^d}{\cst}\right)\drm t 
\\&\qquad + 4S^2\int_{\cst_3\sqrt{2}}^{\infty}t^3\exp\left( - \frac{tkh_k^d}{\cst}\right)\drm t
    \\&\leq  \cst k^{-2}h_k^{-2d}\log(k+1)^2     .
\end{aligned}
\end{equation}
Using \eqref{eq:B4}, \eqref{eq:B6} and the fact that $\epsilon = \text{diam}(\com)h_k^{\frac{d}{2}+1}k^{-\frac{1}{2}}$ we get \eqref{first_sup}.

To prove inequality \eqref{second_sup} we argue analogously to \eqref{eq:lem18_5} to obtain
\begin{align*}
    \Exp\left[\sup_{\bx\in\com}\norm{B_{k,\lambda}(\bx)^{-1}}^{4}_{\op}\right] 
    &\leq 8\lambda_{\min}^{-4}\lambda^{-4}\Exp\left[\sup_{\bx\in\com}\norm{B_{k,\lambda}(\bx)-\Exp\left[B_{k,\lambda}(\bx)\right]}_{\op}^{4}\right] + 8\lambda^{-4}_{\min}
    \\
    &\le \cst\lambda_{\min}^{-4}\lambda^{-4}k^{-2}h_{k}^{-2d}\log(k+1)^2+ 8\lambda^{-4}_{\min},
\end{align*}
where the last inequality follows from \eqref{eq:B6}.
Since, by assumption, $kh_k^d \geq \lambda^{-2}\log(k+1)$ the bound \eqref{second_sup} follows.

Finally, the bound \eqref{third_sup} is derived from \eqref{first_sup} and \eqref{second_sup} by using the same argument as in \eqref{eq:lem18_6} with the only difference that $\norm{\cdot}_{\op}$ should be replaced by $\sup_{\bx\in\com}\norm{\cdot}_{\op}$.
\end{proof}

\begin{lemma}\label{boundsupsubG}
 Let Assumptions \ref{ass:kernel}, \ref{assnoise} and \ref{distass} hold. Then, for $k \in [n]$ and $h_k = \left(\frac{\log(k+1)}{k}\right)^{\frac{1}{2\beta+d}}$ we have
\begin{align*}
    \Exp\left[\sup_{\bx\in\com}\norm{\bG_k(\bx)}^{4} \right]\leq \cst k^{-2}h_k^{-2d}\log(k+1)^2\enspace.
\end{align*}
\end{lemma}
\begin{proof}
Recall that $\bG_k(\bx) = k^{-1}h_k^{-d}\sum_{i=1}^{k}\bR_{i,k}(\bx)\xi_i$ and set $F_{i}^{(s)}(\bx) = \left(\bU\left(\frac{\bx_i-\bx}{h_{k}}\right)K\left(\frac{\bx_i-\bx}{h_{k}}\right)\right)_{s}\xi_i$, for $i\in[k]$, where $\left(\bU\left(\frac{\bx_i-\bx}{h_{k}}\right)K\left(\frac{\bx_i-\bx}{h_{k}}\right)\right)_{s}$ is the $s$-th coordinate of the vector $\bR_{i,k}(\bx)=\bU\left(\frac{\bx_i-\bx}{h_{k}}\right)K\left(\frac{\bx_i-\bx}{h_{k}}\right)$, for $s\in[S]$.  
 Note that 
\begin{align*}
    \norm{\bG_k(\bx)} \leq \sum_{s=1}^{S} \left|k^{-1}h_k^{-d}\sum_{i=1}^{k}F_{i}^{(s)}(\bx)\right|,
\end{align*}
and therefore
\begin{align}\label{G1}
    \Exp\left[\sup_{\bx\in\com}\norm{\bG_k(\bx)}^4\right] \leq S^3\sum_{s=1}^{S}\Exp\left[\sup_{\bx\in\com}\left|k^{-1}h_{k}^{-d}\sum_{i=1}^{k}F_i^{(s)}(\bx)\right|^4\right].
\end{align}
Next, 
arguing as in the proof of Lemma \ref{boundBsup} we deduce that, for any $\bx,\by\in\com$,
\begin{align}\label{G2}
    \left|\sum_{i=1}^{k}\left(F_{i}^{(s)}(\bx) - F_{i}^{(s)}(\by)\right)\right|\leq \cst k h_{k}^{-1}\norm{\bx-\by}.
\end{align}
Given \eqref{G1} and \eqref{G2}, the rest of the proof is quite analogous to the proof of Lemma \ref{boundBsup}. Thus, as in \eqref{eq:B3} we get 
\begin{equation}\label{eq:primebiundsubG}
\begin{aligned}
    \Exp\left[\sup_{\bx\in\com}\norm{\bG_k(\bx)}^4\right] &\leq 
    8 S^3\sum_{s=1}^{S}\Exp\left[\sup_{\bx\in\mathcal{N}}\left|k^{-1}h_{k}^{-d}\sum_{i=1}^{k}F_i^{(s)}(\bx)\right|^4\right] + \cst h_k^{-4d-4}\epsilon^4,
\end{aligned}
\end{equation}
and, for any $t>0$,
\begin{align*}
\Prob\left[\sup_{\bx\in\mathcal{N}}\left|k^{-1}h_k^{-d}\sum_{i=1}^{k}F_i^{(s)}(\bx)\right|\geq t\right]&\leq
\left(\frac{\text{diam}(\com)}{\epsilon}+1\right)^{d}\sup_{\bx\in\mathcal{N}}\Prob\left[\left|k^{-1}h_k^{-d}\sum_{i=1}^{k}F_i^{(s)}(\bx)\right|\geq t\right].
\end{align*}
To control the probability on the right hand side
we apply Lemma \ref{lem:Bern} with $\zeta_i=h_k^{-d}F_{i}^{(s)}(\bx)$. 
For any $q \geq 2$, and $\bx \in \mathcal{N}$ we have
\begin{align}
\begin{aligned}
  \label{subexp1}
\Exp\left[\left|\left(\bU(\frac{\bx_i - \bx}{h_k})K(\frac{\bx_i - \bx}{h_k})\right)_{s}\right|^{q}\right] 
    &= h_k^{d} \int \left|\left(\bU(\bu)K(\bu)\right)_{s}\right|^{q} p(\bx + h_k\bu)\drm \bu
    \\
    &\le   p_{\max} h_k^{d} \int \left|\left(\bU(\bu)K(\bu)\right)_{s}\right|^{q} \drm \bu 
    \leq \cst_0 \cst_1^q h_k^{d},     
\end{aligned}
\end{align}
where we have used the fact that the functions $\left(\bU(\bu)K(\bu)\right)_{s}$ are uniformly bounded on the support of $K$ for all $s\in [S]$. Here, the constants $\cst_0,\cst_1$ do not depend on $\bx,s$.  
Since $\xi_i$ is a sub-exponential random variable we have $\Exp[|\xi_i|^q]\le (\cst_2 q)^{q}$ for any $q\ge1$, where the constant $\cst_2>0$ does not depend on $q$. Using this fact and \eqref{subexp1} we obtain 
\begin{align}\label{subexp1_1}
\Exp\left[\left|F_{i}^{(s)}(\bx)\right|^{q}\right] 
&\leq \cst_0 (\cst_3 q)^{q} h_k^{d}.
\end{align}
It follows from \eqref{subexp1_1} that, for any $q\in  \{2,3,\dots \}$,  
\begin{align*}
\Exp[|\zeta_i|^q]\le \cst_0 (\cst_3 q)^{q} h_k^{-(q-1)d} \le (q!/2) \nu H^{q-2}, \quad  
\end{align*}
where $\nu=\cst_4 h_k^{-d}$, $H = \cst_5 h_k^{-d}$ with some constants $\cst_4,\cst_5$ that do not depend on $q$. Applying Lemma \ref{lem:Bern} we get that,
 for any  $t>0$ and any fixed $\bx$,
\begin{align*}
\Prob\left[\left|k^{-1}h_k^{-d}\sum_{i=1}^{k}F_i^{(s)}(\bx)\right|\geq t\right] 
&\leq 2\exp\left( -\frac{t^2kh_k^d}{\cst(1 + t)}\right),
\end{align*}
so that
\begin{align*}
\Prob\left[\sup_{\bx\in\mathcal{N}}\left|k^{-1}h_k^{-d}\sum_{i=1}^{k}F_i^{(s)}(\bx)\right|\geq t\right] &\leq 
2\exp\left(-\frac{t^2kh_k^d}{\cst(1 + t)} + d\log\left(\frac{\text{diam}(\com)}{\epsilon}+1\right)\right)
      \enspace.
\end{align*}
Note that this inequality is of the same form as \eqref{eq:B5}. The rest of the proof repeats the argument after \eqref{eq:B5} in the proof of Lemma \ref{boundBsup} and therefore we omit it. 
\end{proof}

\begin{lemma}\label{boundCX} Let $k \in [n]$,  $kh_k^d\ge 1$, and let Assumptions \ref{ass:kernel}, \ref{mainassump}(iv) and \ref{distass} hold. Then
\begin{align}\label{eq:lem:boundCX}
   \sup_{\bx\in\com} \Exp \left[\norm{\bC_k(\bx) - \Exp\left[\bC_k(\bx)\right]}^4\right] \leq \cst h_k^{-2d}k^{-2}
\end{align}
and 
\begin{align}\label{eq:lem:boundCX-1}
\sup_{\bx\in\com} \Exp \left[\norm{\bC_k(\bx))}^4\right] \leq \cst.
\end{align}
\end{lemma}

\begin{proof} Recall that $\bC_k(\bx)=k^{-1}h_k^{-d}\sum_{i=1}^{k}\bU\left(\frac{\bx_i-\bx}{h_{k}}\right)K\left(\frac{\bx_i-\bx}{h_{k}}\right)f(\bx_i)$. Assumption  \ref{mainassump}(iv) implies that $f(\bx_i)$'s are uniformly bounded. Given this fact, the proof of \eqref{eq:lem:boundCX} is derived following the same lines as in Lemma \ref{boundB} and it is therefore omitted. The bound \eqref{eq:lem:boundCX-1} is obtained by plugging \eqref{eq:dop2} and \eqref{eq:lem:boundCX} in the inequality 
$$\Exp \left[\norm{\bC_k(\bx)}^4\right]\le 8\left(\Exp \left[\norm{\bC_k(\bx) - \Exp\left[\bC_k(\bx)\right]}^4\right]
+ \norm{\Exp [\bC_k(\bx)]}^4\right)
.
$$\end{proof}

\begin{lemma}\label{boundC} Let $k \in [n]$, and $h_k = \left(\frac{\log(k+1)}{k}\right)^{\frac{1}{2\beta+d}}$. Let Assumptions \ref{ass:kernel}, \ref{mainassump}(iv) and \ref{distass} hold. Then
\begin{align*}
   \Exp \left[\sup_{\bx\in\com}\norm{\bC_k(\bx) - \Exp\left[\bC_k(\bx)\right]}^4\right] \leq \cst h_k^{-2d}k^{-2}\log(k+1)^2.
\end{align*}
\end{lemma}
The proof of this lemma is omitted since, for the same reason as in Lemma \ref{boundCX}, it is  analogous to the proof of Lemma \ref{boundBsup}.

\begin{lemma}\label{covariance}
Let $k\in[n]$, and $h_k = \left(\frac{\log(k+1)}{k}\right)^{\frac{1}{2\beta+d}}$. Let Assumptions \ref{ass:kernel}, \ref{mainassump}(iv) and \ref{distass} hold.  If $kh_k^d\geq \lambda^{-2}\log(k+1)$, then
\begin{align*}
   \sup_{\bx\in\com} \Exp\left[\norm{B_{k,\lambda}(\bx)^{-1} - \left(\Exp[B_{k,\lambda}(\bx)]\right)^{-1}}_{\op}\norm{\bC_k(\bx)-\Exp\left[\bC_k(\bx)\right]}\right]\leq \cst h_k^{-d}k^{-1}\enspace.
\end{align*}
\end{lemma}
This lemma follows immediately from the Cauchy-Schwarz inequality, \eqref{eq:lem18_3} and \eqref{eq:lem:boundCX}.

Finally,  we recall a lemma from \cite{akhavan2020} used in the proof of Theorem \ref{pointest}.
\begin{lemma}[\cite{akhavan2020}  \label{lemm:akhavan_recursive}
Lemma D.1] Let $\{b_k\}$ be a sequence of real numbers such that for all integers $k\geq 2$, 
    \begin{align*}
        b_{k+1}\leq \left(1-\frac{1}{k}\right)b_k+\sum_{i=1}^{N}\frac{a_i}{k^{p_i+1}}\enspace,
    \end{align*}
where $0<p_i<1$ and $a_i\geq 0$ for $1\leq i\leq N$. Then for $k\geq 2$ we have
\begin{align*}
    b_k\leq \frac{2b_2}{k} + \sum_{i=1}^{N}\frac{a_i}{(1-p_i)k^{p_i}}\enspace.
\end{align*}
\end{lemma}

\end{document}